\documentclass[12pt,oneside,openany]{article}
\usepackage{tikz}
\usepackage{tikz-cd}

\usepackage{amssymb}   
\usepackage{amsthm}
\usepackage{amsmath}
\usepackage{amsfonts}
\usepackage{latexsym}
\usepackage{graphics}
\usepackage{fancyhdr}
\usepackage{epstopdf}
\usepackage{setspace}
\usepackage{url}
\usepackage[top=1in, bottom=1in, left=.75in, right=.75in]{geometry}

\newtheorem{theorem}{Theorem}[section]

\newtheorem{prop}[theorem]{Proposition}
\newtheorem{lemma}[theorem]{Lemma}

\newtheorem{corollary}[theorem]{Corollary}
\newtheorem{definition}[theorem]{Definition}

\def\R{{\mathbb R}} 
\def\Z{{\mathbb Z}} 
\def\N{{\mathbb N}} 
\def\Q{{\mathbb Q}} 

\title{Generalizing the Minkowski Question Mark Function to a Family of Multidimensional Continued Fractions}
\author{Thomas Garrity and Peter M. McDonald\\
Department of Mathematics\\
Williams College\\
Williamstown, MA 01267\\
tgarrity@williams.edu}

\begin{document}
\maketitle

\begin{abstract}
The Minkowski question mark function $?:[0,1]\rightarrow [0,1]$ is a continuous, strictly increasing, one-to-one and onto function that has derivative zero almost everywhere.
Key to these facts  are the basic properties of continued fractions.  Thus $?(x)$ is a naturally occurring number theoretic singular function.  This paper generalizes the question mark function to the $216$  triangle partition (TRIP) maps.  These are multidimensional continued fractions which generate a family of almost all known multidimensional continued fractions.  We show for each TRIP map that there is a natural candidate for its analog of the Minkowski question mark function.  We then show that the analog is singular for  $96$ of the TRIP maps and show that $60$ more are singular under an assumption of ergodicity.  
\end{abstract}

\section{Introduction}
The Minkowski question mark function $?:[0,1]\rightarrow [0,1]$ is a continuous, strictly increasing, one-to-one and onto function that has two key properties:
\begin{enumerate}
\item $?(x)$ sends the quadratic irrationals to the rationals.
\item $?(x)$ has derivative zero almost everywhere.
\end{enumerate}
Key to both of these  are the basic properties of continued fractions.  More precisely, it is critical that 
\begin{enumerate}
\item A real number has an eventually periodic continued fraction expansion precisely when it is a quadratic irrational.
\item The partial fractions of a real number's continued fraction expansion provide the best possible  Diophantine approximations to the initial real.
\end{enumerate}

Attempts to find higher dimensional analogs of continued fractions  are called multidimensional continued fraction algorithms, of which there are many.  They have applications in areas ranging from 
   simultaneous Diophantine approximation problems (see Lagarias  \cite{Lagarias93}) to attempts to understand algebraic numbers via periodicity conditions   \cite{GarrityT01} to automata theory (see Fogg \cite{Fogg}).   For background on multidimensional continued fractions, see Schweiger \cite{schweiger} and  Karpenkov \cite{Karpenkov}.

Recently many if not most known multidimensional continued fraction algorithms have been put into a common framework of a family of algorithms, which are called triangle partition maps (TRIP maps) \cite{tripmaps}.  It is natural to consider if there are analogs of the Minkowksi question mark function for each TRIP map, and indeed this is the goal of this paper.  We will see that for each TRIP map there is only one natural candidate for its version of the question mark function.   Among the $216$ TRIP maps, we will see that for $96$ of them, the corresponding function will be singular.  We will also see. subject to ergodicity assumption of the corresponding map, that an additional $60$ will have their corresponding functions be singular.  For the remaining maps, we do not know if the maps are singular.

In the next section, we review the well-known Minkowksi question mark function, mainly to see  how our candidates for analogs are indeed analogs. Then we review the three earlier attempts at such generalizations, in particular the  work of Panti \cite{panti} on the M\"{o}nkemeyer map.  (Note that the   M\"{o}nkemeyer is one of the TRIP maps, a fact we will use in secton \ref{Monkmayer}.)  In section \ref{TRIP}, we will review the basics of triangle partition maps.   For the expert, we will be using, overall,  an additive approach instead of a multiplicative approach (which was the approach in earlier papers).  The triangle partition maps come down to various ways of partitioning a triangle.  These partitionings, we will see, will give us information about the domains of our generalized Minkowski question mark functions.  In section \ref{bary}, we look at what we call barycentric paritionings of a triangle.  There will be a different one corresponding to each of the 216 triangle partition maps.  These barycentric partitionings are what will give us range information about the generalized Minkowski question mark functions. In section \ref{analog} we finally give the definition for our analogs of the question mark function.  We will see that not all of these analogs are actually functions but can instead, sometimes, be maps that take points in a triangle to line segments.  Thus in general, for some triangle partition maps, the analog is a relation. Most of the time, for most TRIP maps, the analog will be a relation, as we discuss in section \ref{convergence}.   Since the analogs to the Minkowski question mark might not be a function, we need to exercise a bit of care as to what we mean by the analog being singular, which we do in section \ref{definition}.  

We do not want to do the painful work of looking at each of the 216 analogs.  In section \ref{reduction}, we show how we can break the 216 different cases into 15 classes, meaning to show singularity of all elements in a class,  we just have to show singularity for a single element.  In section \ref{results}, we show that we do get singularness for five of these classes and show for an additional  five classes we  have singularness subject to an assumption of ergodicity.  We then conclude with a few of the many remaining open questions.

\section{Earlier Work}

\subsection{Minkowski Question Mark Function}

Lagarias said that the Minkowski question mark function is a type of mathematical pun.

To define the Minkowski question mark function, we first define the Farey addition of two rationals.  Given $\frac{p}{q},\frac{r}{s}\in\Q$ define the operation $\hat{+}$ by setting
\[\frac{p}{q}\hat{+}\frac{r}{s}=\frac{p+r}{r+s}\]
Thinking of the fractions $\frac{p}{q}$ and $\frac{r}{s}$ as the vectors $(q,p), (s,r)\in\Z^2$, we notice that this is simply vector addition.  When we take the Farey sum of two rational numbers, we are really adding them as vectors in $\Z^2$ and mapping  back into $\Q$.

Then we define the $n^{th}$ Farey set recursively, beginning  with $\mathcal{F}_0=\{\frac{0}{1},\frac{1}{1}\}$.  Then set
\[\mathcal{F}_1=\left\{\frac{0}{1},\frac{0}{1}\hat{+}\frac{1}{1},\frac{1}{1}\right\}=\left\{\frac{0}{1},\frac{1}{2},\frac{1}{1}\right\}\\\]
Following this pattern, we define $\mathcal{F}_n$ by taking the union of $\mathcal{F}_{n-1}$ and the Farey sums of adjacent elements of $\mathcal{F}_{n-1}$ and ordering them appropriately.  Then 
\begin{align*}
\mathcal{F}_0&=\left\{\frac{0}{1},\frac{1}{1}\right\}\\
\mathcal{F}_1&=\left\{\frac{0}{1},\frac{1}{2},\frac{1}{1}\right\}\\
\mathcal{F}_2&=\left\{\frac{0}{1},\frac{1}{3},\frac{1}{2},\frac{2}{3},\frac{1}{1}\right\}\\
\mathcal{F}_3&=\left\{\frac{0}{1},\frac{1}{4},\frac{1}{3},\frac{2}{5},\frac{1}{2},\frac{3}{5},\frac{2}{3},\frac{3}{4},\frac{1}{1}\right\}
\end{align*}
Similarly, we define the $n^{th}$ barycentric set, $\mathcal{B}_n$, by starting with $\mathcal{B}_0=\{0,1\}$ and creating $\mathcal{B}_n$ by taking the union of $\mathcal{B}_{n-1}$ and the averages of adjacent elements of $\mathcal{B}_{n-1}$ and ordering appropriately.  We present the first few iterations of this:
\begin{align*}
\mathcal{B}_0&=\left\{\frac{0}{1},\frac{1}{1}\right\}\\
\mathcal{B}_1&=\left\{\frac{0}{1},\frac{1}{2},\frac{1}{1}\right\}\\
\mathcal{B}_2&=\left\{\frac{0}{1},\frac{1}{4},\frac{1}{2},\frac{3}{4},\frac{1}{1}\right\}\\
\mathcal{B}_3&=\left\{\frac{0}{1},\frac{1}{8},\frac{1}{4},\frac{3}{8},\frac{1}{2},\frac{5}{8},\frac{3}{4},\frac{7}{8},\frac{1}{1}\right\}
\end{align*}
Notice that $\mathcal{F}_0=\mathcal{B}_0$ and $\mathcal{F}_1=\mathcal{B}_1$ but that after this point they diverge.  Now we are ready to define $?(x)$ as a limit of functions we will call $?_n(x).$

Define $?_0(x)=x$ on $[0,1]$.  This corresponds to identifying the first element of $\mathcal{F}_0$ with the first element of $\mathcal{B}_0$ and the second element of $\mathcal{F}_0$ with the second element of $\mathcal{B}_0$ and connecting these points with a straight line.  Then we can define $?_n(x)$ by $?_n(f_{i,n})=b_{i,n}$, where $f_{i,n}$ is the $i^{th}$ element of $\mathcal{F}_n$ and $b_{i,n}$ is the $i^{th}$ element of $\mathcal{B}_n$, and filling in the rest by linearity. This means each $?_n(x)$ will be a piecewise linear function.  The first few $?_i$ are shown below:

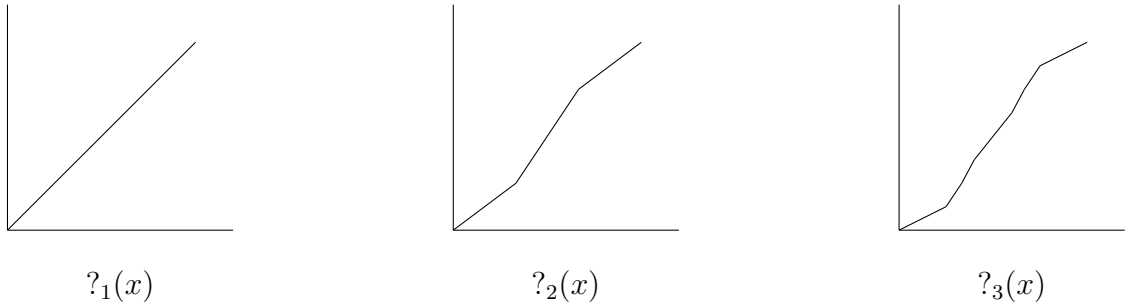
\begin{align*}
\begin{tikzpicture}
\draw (0,0)--(3,0);
\draw (0,0)--(0,3);
\draw (0,0)--(2.5,2.5);
\node at(1.5,-0.75) {$?_1(x)$};
\end{tikzpicture}&&&&
\begin{tikzpicture}
\draw (0,0)--(3,0);
\draw (0,0)--(0,3);
\draw (0,0)--(5/6,5/8);
\draw (5/6,5/8)--(5/4,5/4);
\draw (5/4,5/4)--(10/6,15/8);
\draw (10/6,15/8)--(2.5,2.5);
\node at(1.5,-0.75){$?_2(x)$};
\end{tikzpicture}&&&&
\begin{tikzpicture}
\draw (0,0)--(3,0);
\draw (0,0)--(0,3);
\draw (0,0)--(5/8,5/16);
\draw (5/8,5/16)--(5/6,5/8);
\draw (5/6,5/8)--(1,15/16);
\draw (1,15/16)--(5/4,5/4);
\draw (5/4,5/4)--(1.5,25/16);
\draw (1.5,25/16)--(10/6,15/8);
\draw (10/6,15/8)--(15/8,35/16);
\draw (15/8,35/16)--(2.5,2.5);
\node at(1.5,-0.75) {$?_3(x)$};
\end{tikzpicture}
 \end{align*}
Then we define
\[\lim_{n\to\infty}?_n(x)=?(x).\]
It can be shown that this convergence is uniform, and hence that $?(x)$ is continuous.  What is important is that it has derivative zero almost everywhere.
Salem \cite{salem} shows that the Question Mark Function is singular by showing that everywhere the derivative exists and is finite, the derivative is zero.  Since the derivative must exist and be finite almost everywhere, we have that $?(x)$ is singular.  A further discussion of alternate definitions and explorations of the singularity of $?(x)$ can also be found in \cite{viader}. In looking for higher dimensional singular functions we will distill Salem's argument to its essential ingredients by using the limit definition of $?(x)$ .  Each $?_n$ can be though of as a map from $I_{n,F}$, the unit interval partitioned according to $\mathcal{F}_n$, to $I_{n,B}$, the unit interval partitioned according to $\mathcal{B}_n$. Because $?_n$ is piecewise linear, the derivative at any point $\alpha\in(f_{j,n},f_{j+1,n})$ is simply 
\[\frac{b_{j+1,n}-b_{j,n}}{f_{j+1,n}-f_{j,n}}.\]  


As $n$ goes to infinity, this ratio goes to zero almost everywhere.  Thus the singularity of the Question Mark Function is due to the relationship between the Farey and barycentric partitions of the unit interval.  The Farey subdivision is intimately linked to the continued fraction algorithm, which gives the best Diophantine approximation of a real number.  Thus, it would make sense to look for generalizations of the Minkowski Question Mark Function by looking at multidimensional fraction algorithms on the triangle in $\R^2$ as a sort of Farey partitioning and developing a notion of barycentric partitioning associated to each one. 

\subsection{Panti's Work}
The first attempt at generalizing the Minkowski Question Mark Function was done in \cite{beaver} and follows a similar line of reasoning to explaining the singularity of $?(x)$.    That paper looked at only one type of multi-dimensional continued fraction, the Farey multi-dimensional continued fraction algorithm.  This work was extended by Marder \cite{marder}.
The third attempt at generalizing the Minkowski Function is presented  by Panti \cite{panti}.  Here the multi-dimensional continued fraction that is used is the  M\"{o}nkemeyer map.   In section 1 of that paper, Panti notes that the Question Mark Function conjugates the Farey map
\[F(x)=\begin{cases} \frac{x}{1-x}&\text{if } 0\leq x<1/2\\\frac{1-x}{x}&\text{if } 1/2\leq x\leq 1\end{cases}\]
with the tent map
\[T(x)=\begin{cases}2x&\text{if } 0\leq x<1/2\\2-2x&\text{if } 1/2\leq x\leq 1\end{cases}\]
and proves in Proposition 1.1 that this property characterizes $?(x)$.  Thus, Panti turns to the $n$-dimensional generalizations of these maps, known as the M\"{o}nkemeyer map ($M$) and the tent map ($T$) to find his generalization of $?(x)$.  (At this point the reader should note that Dasratha et al. showed in \cite{tripmaps} that the M\"{o}nkemeyer map in the 2-dimensional case corresponds to the TRIP map $T_{(e,132,23)}$, a fact that we will later be exploiting.)  Defining $\triangle$ as the $n$-dimensional simplex that reduces to our familiar triangle in 2-dimensions, he presents the following
\begin{theorem} There exists a unique homeomorphism $\Phi:\triangle\to\triangle$  such that $T=\Phi\circ M\circ\Phi^{-1}$.
\end{theorem}
Proving that this homeomorphism exists and is unique is the brunt of his paper and comes largely from the fact that both $M$ and $T$ give partitions of $\triangle$ that converge to a point. Because of this, each point of $\{0,1\}^{\N}$ corresponds to a unique point in $\triangle$ under an appropriate ordering.  Then define 
\begin{align*}
\varphi&:\{0,1\}^\N\to\triangle_F\\
\nu&:\{0,1\}^\N\to\triangle_B
\end{align*}
by sending the sequence $(i_1,i_2,\dots)$ to appropriate point in $\triangle_F$ and $\triangle_B$ respectively.  It turns out that $\varphi$ and $\nu$ are continuous, surjective, and have the same fibers, so we can define an equivalence relation on $\{0,1\}^\N$ by $a\equiv b$ if and only if $\varphi(a)=\varphi(b)$ if and only if $\nu(a)=\nu(b)$.  Then we can define bijections $\bar{\varphi},\bar{\nu}:{0,1}^\N/\equiv\to\triangle$ as the natural quotient mappings of $\varphi$ and $\nu$ respectively.  Then we define $\Phi$ as follows:
\begin{definition} We define $\Phi:\triangle\to\triangle$ as the homeomorphism $\Phi=\bar{\nu}\circ\bar{\varphi}^{-1}$.  Equivalently, $\Phi(p)=\nu(a)$ for any $a$ such that $\varphi(a)=p$.
\end{definition}
With this definition of $\Phi$ we get the following commutative diagram with 
\begin{eqnarray*}
M\circ \phi &=& \phi \circ S  \\
T\circ \nu &=& \nu \circ S  \\
T\circ \Phi &=& \Phi \circ M
\end{eqnarray*} 
where $S$ denotes the shift map on $\{0,1\}^\N$ (i.e. $S(a_0a_1a_2\dots)=(a_1a_2a_3\dots)$):

$$\begin{array}{ccc}
\triangle & \overset{M}{\longrightarrow} & \triangle \\
&&\\
\phi \uparrow & & \uparrow \phi  \\
&&\\
\{0,1\}^N & \longrightarrow & \{0,1\}^N\\
&&\\
\nu \downarrow & & \downarrow \nu \\
&&\\
\triangle &  \overset{T}{\longrightarrow}& \triangle 
\end{array}$$


This is another example of a Lagarias type pun.  The above diagram provides good foreshadowing for much of this paper.  We would have liked to find analogs of the above for other multidimensional continued fractions.  Unfortunately, only partial analogs exist.  For most of the multidimensional continued fractions that we will be concerned with, the maps $\phi$ and $\nu$ do not exist, though the analogs of their inverses do.  This will explained further in section \ref{convergence}.

Now back to Panti's work on  the singularity of $\Phi$ with respect to the Lebesgue measure $\lambda$. We normalize $\lambda$ so that $\lambda(\triangle)=1$.  Panti defines a probability measure $\mu$ on $\triangle$ that is induced by 
\[h(x_1,\dots,x_n)=\frac{1}{x_1(x_1-x_2+1)(x_1-x_3+1)\cdots(x_1-x_n+1)}\]
and properly normalizing $\mu$ we have that the M\"{o}nkemeyer map $M$ preserves $\mu$ and is ergodic with respect to it.  It can also be shown that $T$ is ergodic with respect to $\lambda$ and thus that $M$ is ergodic with respect to the measure $\Phi_*^{-1}\lambda$ where $\Phi_*^{-1}\lambda(A)=\lambda(\Phi(A))$.  It turns out that $\mu$ and $\Phi_*^{-1}\lambda$ are different but both ergodic with respect to the same transformation $M$.  This means that they are mutually singular, which means by standard work that   $\Phi$ is singular with respect to $\lambda$.

\section{Triangle Partition Maps}\label{TRIP}

\subsection{Basic Notation}
Our triangle will be 
$$\triangle = \{(x,y) \in \R^2: 1>x>y>0\},$$
which has vertices $(0,0)$, $(1,0)$ and $(1,1)$.  We will regularly identify a point in $\R^3$ with a point in $\R^2$ via the map
$$\pi:\R^3 -(x\neq 0) \rightarrow \R^2$$
given by 
$$\pi   \left(  \begin{array}{c}  x \\ y \\ z  \end{array}   \right)  = \left( \frac{y}{x}, \frac{z}{x} \right).$$
Thus $\pi$ sends points on a ray in $\R^3 $ to a point in $\R^2$. This is why we will regularly  identify the triangle $\triangle$ to  the cone in $\R^3$ spanned by 
$$v_1 = \left(  \begin{array}{c}  1 \\ 0 \\ 0  \end{array}   \right) , v_2 = \left(  \begin{array}{c}  1 \\ 1 \\ 0  \end{array}   \right) , v_3  = \left(  \begin{array}{c}  1 \\ 1 \\ 1 \end{array}   \right) . $$
Note that 
$$\pi(v_1) = (0,0). \pi(v_2) = ( 1,0), \pi(v_3) = (1,1).$$
Set 
$$V= ( v_1, v_2, v_3) =  \left(  \begin{array}{ccc}  1 & 1 & 1 \\ 0 & 1 & 1 \\ 0 & 0 & 1  \end{array}   \right) .$$

\subsection{The Initial Triangle Partition}
The classical question mark function involves two different systematic methods for subdividing the unit interval, which is of course a one dimensional simplex.  Our generalizations will involve two different systematic ways for subdividing a triangle, which is a two dimensional simplex.  One of the ways will be number theoretic, and has already been used in \cite{Stern}, which is why this section has heavy overlap with that earlier paper.  Also, many of the diagrams below are similar to diagrams in \cite{Stern}.
We start with setting 
$$F_0 =  \left(  \begin{array}{ccc}  0 & 0 & 1 \\ 1& 0 & 0 \\ 0 & 1 & 1  \end{array}   \right)\; \mbox{and} \;F_1 =  \left(  \begin{array}{ccc}  1 & 0 & 1 \\ 0& 1 & 0 \\ 0 & 0 & 1  \end{array}   \right) .$$
(In the notation in the next section, $F_0$ will be $F_0(e,e,e)$ and $F_1$ will be $F_1(e,e,e)$.)
We have that 
$$VF_0= (v_1, v_2, v_3)  \left(  \begin{array}{ccc}  0 & 0 & 1 \\ 1& 0 & 0 \\ 0 & 1 & 1  \end{array}   \right) = (v_2,v_3, v_1+v_2) =  \left(  \begin{array}{ccc}  1 & 1 & 2 \\ 1& 1 & 1 \\ 0 & 1 & 1  \end{array}   \right) $$
and
$$VF_1= (v_1, v_2, v_3)  \left(  \begin{array}{ccc}  1 & 0 & 1 \\ 0& 1 & 0 \\ 0 & 0 & 1  \end{array}   \right) = (v_1,v_2, v_1+v_2) =  \left(  \begin{array}{ccc}  1 & 1 & 2 \\ 0& 1 & 1 \\ 0 & 0 & 1  \end{array}   \right). $$
We set 
$$\triangle(0) = \pi (VF_0), \; \triangle(1) = \pi (VF_1),$$
subtriangles of the original triangle $\triangle.$ The vertices of $\triangle(0)$ are 
$$(1,0), (1,1), (1/2,1/2)$$
and the vertices of $\triangle(1)$ are 
$$(0,0), (1,0), (1/2,1/2)   .  $$

Pictorially, the initial triangle $\triangle$ is 

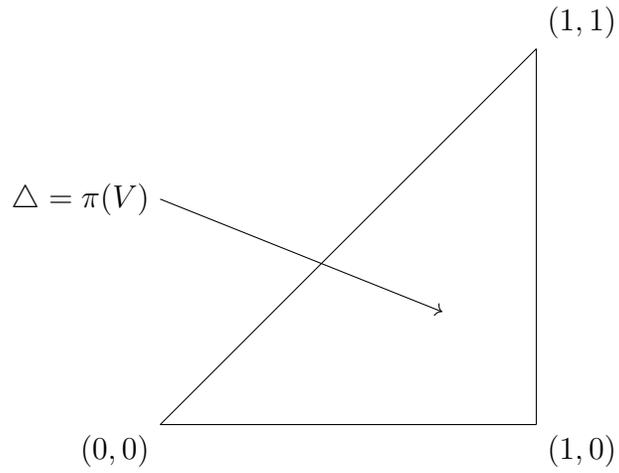
\begin{center}  
\begin{tikzpicture}[scale=5]
\draw(0,0)--(1,0);
\draw(0,0)--(1,1);
\draw(1,0)--(1,1);
\draw[->](0,.6)node[left]{$\triangle = \pi(V)$}--(3/4, .3);

\node[below left]at(0,0){$(0,0)$};
\node[below right]at(1,0){$(1,0)$};
\node[above right]at(1,1){$(1,1)$};
\end{tikzpicture}
\end{center}

The two subtriangles are

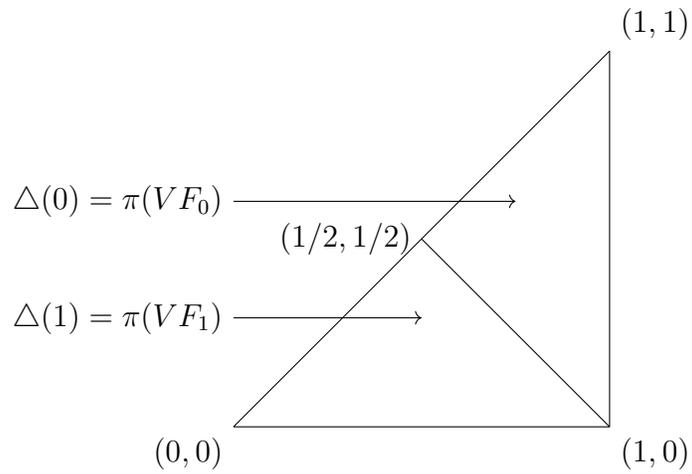
\begin{center}  
\begin{tikzpicture}[scale=5]
\draw(0,0)--(1,0);
\draw(0,0)--(1,1);
\draw(1,0)--(1,1);
\draw(1,0)--(1/2,1/2);
\draw[->](0,.6)node[left]{$\triangle(0)= \pi(VF_0)$}--(3/4,.6);

\draw[->](0,.29)node[left]{$\triangle(1) = \pi(VF_1)$}--(.5,.29);

\node[below left]at(0,0){$(0,0)$};
\node[below right]at(1,0){$(1,0)$};
\node[above right]at(1,1){$(1,1)$};
\node[left]at(1/2,1/2){$(1/2,1/2)$};
\end{tikzpicture}
\end{center}

We continue subdividing, setting 
$$\triangle(00) = \pi ( VF_0 F_0 ), \;\triangle(01) = \pi ( VF_0 F_1 ), \;\triangle(10)= \pi ( VF_1 F_0 ),\; \triangle(11) = \pi ( VF_1 F_1 ).$$
Using that 
$$\begin{array}{ccccccc}
VF_0F_0  &=& (v_2, v_3, v_1+ v_3) F_0 &=& (v_3, v_1 + v_3 , v_1+ v_2 + v_3) & = & \left(  \begin{array}{ccc}  1 & 2 & 3 \\ 1& 1 & 2 \\ 1 & 1 & 1  \end{array}   \right)  \\
VF_0F_1 &=& (v_2, v_3, v_1+ v_3) F_1 &=& (v_2,  v_3 , v_1 +  v_2 + v_3)  & = & \left(  \begin{array}{ccc}  1 & 1 & 3 \\ 1& 1 & 2 \\ 0 & 1 & 1  \end{array}   \right)  \\ 
VF_1F_0 &=& (v_1, v_2, v_1+ v_3) F_0 &=& (v_2, v_1 + v_3 , 2v_1   + v_3)   & = & \left(  \begin{array}{ccc}  1 & 2 & 3 \\ 1& 1 & 1 \\ 0 & 1 & 1  \end{array}   \right) \\
VF_1F_1  &=& (v_1, v_2, v_1+ v_3) F_1 &=& (v_1, v_2 , 2v_1   + v_3)  & = & \left(  \begin{array}{ccc}  1 & 1 & 3 \\ 0& 1 & 1 \\ 0 & 0 & 1  \end{array}   \right) 
\end{array}$$
we have that the vertices of $\triangle(00)$ are
  $$(1,1), (1/2, 1/2), (2/3, 1/3),$$  
  the vertices of $\triangle(01)$ are
  $$(1,0), (1,1), (2/3, 1/3),$$  
  the vertices of $\triangle(10)$ are
  $$(1,0), (1/2, 1/2), (1/3, 1/3),$$  
  and the vertices of $\triangle(11)$ are
  $$(0,0), (1,0), (1/3, 1/3).$$   
   Pictorially we have   
     
     
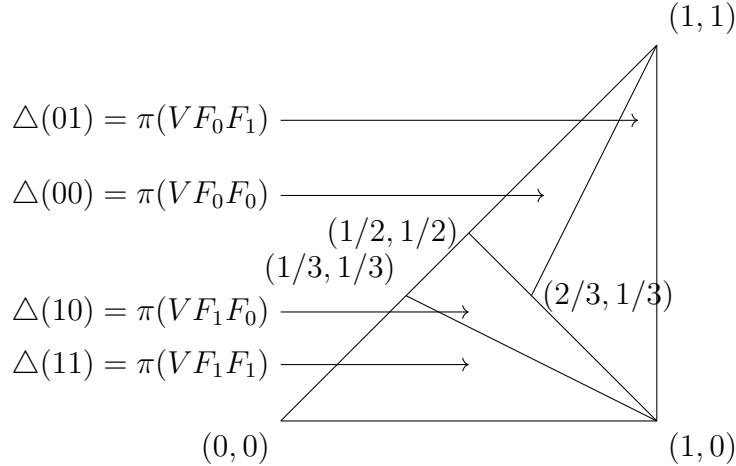
\begin{center}  
\begin{tikzpicture}[scale=5]
\draw(0,0)--(1,0);
\draw(0,0)--(1,1);
\draw(1,0)--(1,1);
\draw(1,0)--(1/2,1/2);
\draw(2/3,1/3)--(1,1);
\draw(1,0)--(1/3,1/3);

\draw[->](0,.8)node[left]{$\triangle(01)= \pi(VF_0F_1)$}--(.95,.8);

\draw[->](0,.6)node[left]{$\triangle(00)= \pi(VF_0F_0)$}--(.7,.6);

\draw[->](0,.29)node[left]{$\triangle(10)= \pi(VF_1F_0)$}--(.5,.29);

\draw[->](0,.15)node[left]{$\triangle(11) = \pi(VF_1F_1)$}--(.5,.15);

\node[below left]at(0,0){$(0,0)$};
\node[below right]at(1,0){$(1,0)$};
\node[above right]at(1,1){$(1,1)$};
\node[left]at(1/2,1/2){$(1/2,1/2)$};
\node[right]at(2/3,1/3){$(2/3,1/3)$};
\node[above left]at(1/3,1/3){$(1/3,1/3)$};
\end{tikzpicture}
\end{center}

We can continue this process, getting triangles 
$$\triangle (i_1i_2 \ldots i_n) = \pi (V F_{i_1}F_{i_2} \cdots F_{i_n}). $$

\begin{definition} A element $(\alpha, \beta) \in \triangle$ has triangle sequence $(i_1, i_2, i_3, \ldots )$, where each $i_j$ is either zero or one, if, for all $n$,
$$(\alpha, \beta) \in  \triangle(i_1, i_2, i_3, \ldots, i_n ).$$
\end{definition}

Part of the number theoretic uses of a number's triangle sequence, proven in \cite{GarrityT01, GarrityT05}  can be seen in 
\begin{theorem}  Suppose that $(\alpha, \beta) \in \triangle$ has an eventually periodic triangle sequence.  Provided that there are infinitely many zeros, we know that $\alpha$ and $\beta$ are  in the same number field of degree less than or equal to three.
\end{theorem}
Thus periodicity implies that $\alpha$ and $\beta$ are at worse cubic irrationals.  

There is one small technical detail.  What happens for a point $(\alpha, \beta) \in \triangle$ on a boundary of one of the $\triangle(i_1, i_2, i_3, \ldots, i_n ).$  This provides an ambiguous definition for the triangle sequence.  There are standard conventions for choosing which sequence to associate to $(\alpha, \beta)$, but since the boundaries are all on a set of measure zero, we will ignore this issue throughout this paper.  It makes no difference to any of the results.

\subsection{Triangle Partition Maps}
Our interpretation of the triangle map via three-by-three matrices depends on our choice of column vectors $v_1,v_2,v_3. $   We will get 215 more multidimensional continued fractions by initially permuting the vectors $v_1,v_2,v_3 $ by an element $\sigma$ of the permutation group $S_3$, applying the matrices $F_0$ or $F_1$ and then applying another element of $S_3$.   Write the six elements of $S_3$ as the three-by-three matrices 
\[(12)=\left(\begin{array}{ccc}
0&1&0\\
1&0&0\\
0&0&1\end{array}\right)\qquad
(13)=\left(\begin{array}{ccc}
0&0&1\\
0&1&0\\
1&0&0\end{array}\right)\qquad
(23)=\left(\begin{array}{ccc}
1&0&0\\
0&0&1\\
0&1&0\end{array}\right)\]
\[(123)=\left(\begin{array}{ccc}
0&1&0\\
0&0&1\\
1&0&0\end{array}\right)\qquad
(132)=\left(\begin{array}{ccc}
0&0&1\\
1&0&0\\
0&1&0\end{array}\right)\]

\begin{definition}
For  $    (\sigma, \tau_0, \tau_1)     \in S_3^3$, define
\[F_0(\sigma, \tau_0, \tau_1)=\sigma A_0\tau_0 \text{ and } F_1(\sigma, \tau_0, \tau_1)=\sigma A_1\tau_1.\]
\end{definition}
Each choice of  $(\sigma, \tau_0, \tau_1)\in S_3^3$ gives rise to a new partitioning of the triangle $\triangle,$ by simply setting
$$\triangle_{(\sigma, \tau_0, \tau_1)}(0) = \pi (VF_0(\sigma, \tau_0, \tau_1)) \; \mbox{and}  \; \triangle_{(\sigma, \tau_0, \tau_1)}(0) = \pi (VF_1(\sigma, \tau_0, \tau_1))$$
By iterating, we define for any sequence $(i_1, i_2, i_3, \ldots , i_n)$ of zeros and ones the triangle
$$\triangle_{(\sigma, \tau_0, \tau_1)}(i_1, i_2, i_3, \ldots , i_n) = \pi (V    F_{i_1     }(\sigma, \tau_0, \tau_1)   F_{i_2     }(\sigma, \tau_0, \tau_1)    \cdots F_{i_n     }(\sigma, \tau_0, \tau_1)     ) .$$
For $   (\sigma, \tau_0, \tau_1)   = (e,e,e)$, we get the partition in the previous section.  

\begin{definition}
For any $    (\sigma, \tau_0, \tau_1)     \in S_3^3$, we define the $n$th Farey partition $\mathcal{F}_n (\sigma, \tau_0, \tau_1) $  of the triangle $\triangle$ as the partitioning  given by the subtriangles $\triangle_{(\sigma, \tau_0, \tau_1)}(i_1, i_2, i_3, \ldots , i_n) $.

\end{definition} 
We will denote the limit partition as $\mathcal{F}( \sigma. \tau_0, \tau_1)$ and will say that this is the TRIP map $(\sigma, \tau_0, \tau_1).$

\begin{definition} A element $(\alpha, \beta) \in \triangle$ has $ (\sigma, \tau_0, \tau_1) $ TRIP sequence $(i_1, i_2, i_3, \ldots )$, where each $i_j$ is either zero or one, if, for all $n$,
$$(\alpha, \beta) \in \triangle_{(\sigma, \tau_0 , \tau_1)}  (i_1, i_2, i_3, \ldots, i_n ).$$
(We will sometimes call this the Farey-$ (\sigma, \tau_0, \tau_1) $ TRIP sequence $(i_1, i_2, i_3, \ldots ) $ for $(\alpha, \beta)$.)
\end{definition}
Before looking at one of the number theoretic implications of these sequences, let us look at an example, say $ ((12),(13), e)   .$
Then we have 
\begin{eqnarray*}
VF_0 ((12),(13), e)   & = & V (12) F_0 (13) \\
&=& (v_1, v_2, v_3)  (12) F_0 (13) \\
&=& (v_2, v_1, v_3) F_0 (13) \\
&=&  (v_2, v_1, v_3)  \left(  \begin{array}{ccc}  0 & 0 & 1 \\ 1& 0 & 0 \\ 0 & 1 & 1  \end{array}   \right) (13) \\
&=& (v_1, v_3, v_2 + v_3)(13)  \\
&=& (v_2+v_3, v_3, v_1)  \\
&=&  \left(  \begin{array}{ccc}  2 & 1 & 1 \\ 2& 1 & 0 \\ 1 & 1 & 0  \end{array}   \right).
\end{eqnarray*}
Then the subtriangle $\triangle_{((12),(13), e)}(0)$ will have vertices
$$(1, 1/2), \; (1,1) , \; (0,0).$$
In similar fashion
$$VF_1 ((12),(13), e) = V (12) F_1 e = (v_2, v_1, v_2 + v_3) = \left(  \begin{array}{ccc}  1 & 1 & 2 \\ 1& 0 & 2 \\ 0 & 0 & 1  \end{array}   \right).$$
The subtriangle $\triangle_{((12),(13), e)}(1)$ will have vertices
$$(1, 0), \; (0,0) , \; (1, 1/2).$$
Pictorially, these new subtriangles are

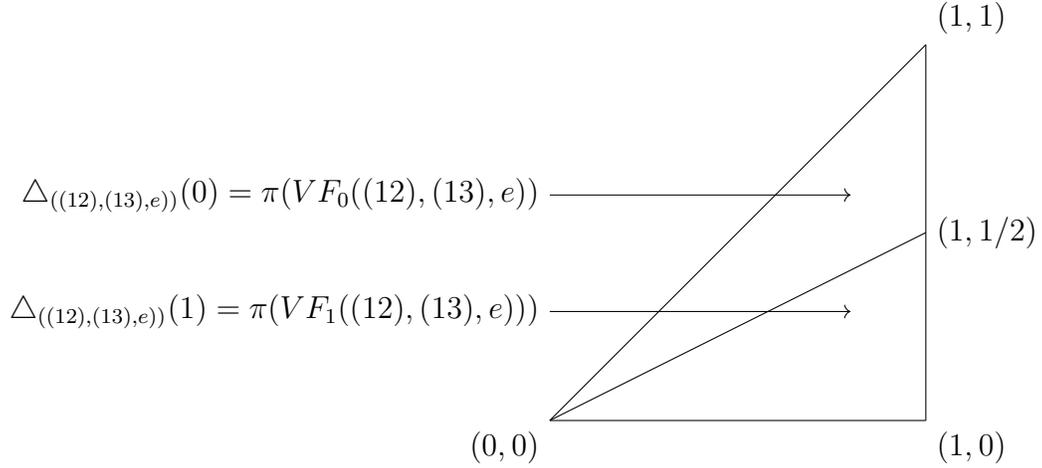
\begin{center}  
\begin{tikzpicture}[scale=5]
\draw(0,0)--(1,0);
\draw(0,0)--(1,1);
\draw(1,0)--(1,1);
\draw(0,0)--(1,1/2);
\draw[->](0,.6)node[left]{$\triangle_{ ((12),(13), e))}       (0)= \pi(VF_0   ((12),(13), e))$}--(4/5,.6);

\draw[->](0,.29)node[left]{$\triangle_{ ((12),(13), e))}    (1) = \pi(VF_1 ((12),(13), e)))$}--(4/5,.29);

\node[below left]at(0,0){$(0,0)$};
\node[below right]at(1,0){$(1,0)$};
\node[above right]at(1,1){$(1,1)$};
\node[right]at(1,1/2){$(1,1/2)$};
\end{tikzpicture}
\end{center}
Let us do one more iteration.  We have
\begin{eqnarray*}
V F_0 ((12),(13), e)  F_0 ((12),(13), e) &=& (v_2+v_3, v_3, v_1) F_0 ((12),(13), e)  \\
&=& (v_1 + v_3, v_1 , v_2 + v_3) \\
&=&   \left(  \begin{array}{ccc}  2 & 1 & 2 \\ 1& 0 & 2 \\ 1 & 0 & 1  \end{array}   \right)   \\
VF_0 ((12),(13), e)  F_1 ((12),(13), e) &=& (v_2+v_3, v_3, v_1) F_1 ((12),(13), e) \\
&=& (v_3, v_2 + v_3, v_1 + v_3)  \\
&=&  \left(  \begin{array}{ccc}  1 & 2 & 2 \\ 1& 2 & 1 \\ 1 & 1 & 1  \end{array}   \right)    \\
VF_1 ((12),(13), e)  F_0 ((12),(13), e) &=& (v_2, v_1, v_2+v_3) F_0 ((12),(13), e)   \\
&=& (v_1+v_2+v_3, v_2 + v_3, v_2)   \\
&=&   \left(  \begin{array}{ccc}  3 & 2 & 1 \\ 2  & 2 & 1 \\ 1 & 1 & 0  \end{array}   \right) \\
VF_1 ((12),(13), e)F_1 ((12),(13), e)&=&      (v_2, v_1, v_2+v_3)           F_1 ((12),(13), e)  \\
&=&  (v_1, v_2, v_1 + v_2 + v_3)  \\
&=&   \left(  \begin{array}{ccc}  1 & 1 & 3 \\ 0& 1 & 2 \\ 0 & 0 & 1  \end{array}   \right)
\end{eqnarray*}
Then the vertices of $\triangle_{((12),(13), e)}(00)$ are
$  (1/2, 1/2)$, $(0,0)$, $(1, 1/2),$
the vertices of $\triangle_{((12),(13), e)}(01)$ are
$ ( 1,1)$, $ (1, 1/2) $, $ (1/2, 1/2),$
the vertices of $\triangle_{((12),(13), e)}(10)$ are
$ (2/3, 1/3)$, $ (1, 1/2)$, $( 1,0),$
and the vertices of $\triangle_{((12),(13), e)}(11)$ are
$ (0,0)$, $  (1,0)$, $ (2/3, 1/3). $
Pictorially we have

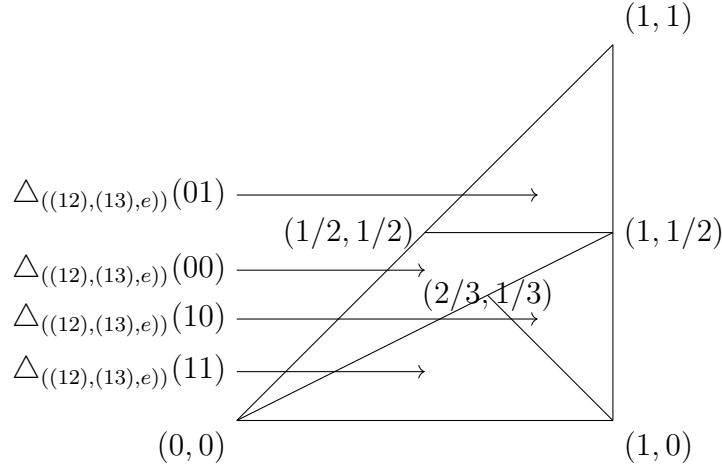
\begin{center}  
\begin{tikzpicture}[scale=5]
\draw(0,0)--(1,0);
\draw(0,0)--(1,1);
\draw(1,0)--(1,1);
\draw(0,0)--(1,1/2);
\draw(1/2,1/2)--(1,1/2);
\draw(2/3,1/3)--(1,0);

\draw[->](0,.6)node[left]{$\triangle_{ ((12),(13), e))}       (01)$}--(4/5,.6);
\draw[->](0,.4)node[left]{$\triangle_{ ((12),(13), e))}       (00)$}--(.5,.4);

\draw[->](0,.27)node[left]{$\triangle_{ ((12),(13), e))}    (10) $}--(4/5,.27);

\draw[->](0,.13)node[left]{$\triangle_{ ((12),(13), e))}    (11) $}--(1/2,.13);

\node[below left]at(0,0){$(0,0)$};
\node[below right]at(1,0){$(1,0)$};
\node[above right]at(1,1){$(1,1)$};
\node[right]at(1,1/2){$(1,1/2)$};
\node[left]at(1/2,1/2){$(1/2,1/2)$};
\node at(2/3,1/3){$(2/3,1/3)$};
\end{tikzpicture}
\end{center}

\subsection{Additive versus multiplicative}\label{multiplicative}

The above development of the triangle sequence is called the additive approach.  In \cite{GarrityT01, GarrityT05, tripmaps, SMALL11q3, SchweigerF08}, the multiplicative version is used.  In the multiplicative version, we look at the subtriangles 
$$\triangle_k ( \sigma. \tau_0, \tau_1)= \pi(VF_1^k( \sigma. \tau_0, \tau_1)F_0( \sigma. \tau_0, \tau_1),$$   where $k$ is a non-negative integer.   Then the map
$$T^G( \sigma. \tau_0, \tau_1):\triangle \rightarrow \triangle,$$
 for an $(x,y)$ in the interior of $\triangle_k( \sigma. \tau_0, \tau_1)$, is

$$T^G( \sigma. \tau_0, \tau_1)(x,y) = \pi \left(    V(VF_1 (\sigma , \tau_0, \tau_1 ) ^kF_0   (\sigma , \tau_0, \tau_1 )  )^{-1} \left(   \begin{array}{c} 1\\ x  \\ y \end{array}\right)   \right).$$
(The superscript $``G''$ is for ``Gauss,'' as this is a generalization of the classical Gauss map for continued fractions.)  We can associate to each $(x,y) \in \triangle$ the sequence of non-negative integers $(a_1, a_2, \dots )$ if 
$$(x,y) \in \triangle_{a_1}( \sigma. \tau_0, \tau_1), T^G(x,y) \in \triangle_{a_2}( \sigma. \tau_0, \tau_1), T^G(T^G(x,y)) \in \triangle_{a_3}( \sigma. \tau_0, \tau_1), \ldots $$
(Here we are suppressing the ``$( \sigma. \tau_0, \tau_1)$'' after each of the maps $T$, for ease of notation.)
We call this sequence the multiplicative $( \sigma. \tau_0, \tau_1)$ sequence, or sometimes the multiplicative Farey $( \sigma. \tau_0, \tau_1)$ sequence.

We can in straightforward way translate between the multiplicative version and the additive version, as we simply take any non-negative $k$ and replace it by $k$ ones followed by a single $0$.  For example, the multiplicative 

$$(2, 3,1,0,2,\ldots )$$
is additively 
$$(1101110100110 \ldots).$$
More precisely, if a pair $(\alpha, \beta)$ has multiplicative sequence $(a_1, a_2, \dots )$, the corresponding additive sequence will be  $a_1$ ones, followed by a zero, then $a_2$ ones, followed by a zero, etc.

It is easier to use the additive version for the definition of our analogs for the Minkowski question mark function.  In the proofs, as we will see, it will be useful to use the multiplicative version.

\subsection{Periodicity, Cubic Irrationals and Convergence Questions}
Each $(\sigma, \tau_0, \tau_1)$ is a different multidimensional continued fraction algorithm. 
Theorem 6.2  in \cite{tripmaps}  is
\begin{theorem} \label{cubic} Suppose that $(\alpha, \beta) \in \triangle$ has an eventually periodic $(\sigma, \tau_0, \tau_1) $  TRIP sequence $(i_0. i_1, i_2, \ldots $ and that the nested sequence of triangles $\triangle_{(\sigma, \tau_0, \tau_1)}  (   i_0  ) \supset \triangle_{(\sigma, \tau_0, \tau_1)}  (   i_0 i_1 )   \supset \triangle_{(\sigma, \tau_0, \tau_1)}  (   i_0 i_1 i_2  ) \supset \cdots  $   converge to a point.  Then $\alpha$ and $\beta$ are  in the same number field of degree less than or equal to three.
\end{theorem}
Thus again  periodicity implies that $\alpha$ and $\beta$ are at worse cubic irrationals.  

For a given  $(\sigma, \tau_0, \tau_1)$, we are far from guaranteed that the nested $\triangle_{(\sigma, \tau_0, \tau_1)}  (   i_0  ) \supset \triangle_{(\sigma, \tau_0, \tau_1)}  (   i_0 i_1 )   \supset \triangle_{(\sigma, \tau_0, \tau_1)}  (   i_0 i_1 i_2  ) \supset \cdots  $ will converge to a point, though we do know that the convergence can be no worse than a line segment.  Also in \cite{tripmaps}, Proposition 7.9 explicitly lists the 156 triangle partition maps for which there are periodic sequences whose corresponding nested sequence of triangles do not converge to a point and Propostion 7.11 explicitly shows for the remaining 60 triangle partition  maps that periodic sequences must have their nested sequence of triangles converge to a point. 

Convergence of the nested sequence  of triangles in general, though, seems hard.  When we have convergence and when we do not is known for the $(e,e, e) $ case \cite{GarrityT05}. In Section \ref{Monkmayer}, we will how that the nested sequence of  triangles will always converge to a point for $24$ of the  triangle partition algorithms.   In Section \ref{degenerate}, we show that for a different $24$ triangle partition algorithms, no nested sequence of the triangles will converge to a point.  Beyond that, nothing is really known.

\section{Barycentric Triangle Partition Maps}\label{bary}

Instead of starting with $F_0$ and $F_1$ as  above, we  now define matrices $G_0$  and $G_1$ by replacing the 1's in the third columns of $F_0$  and $F_1$ with $1/2$s:
\begin{align*}
G_0&=\left(\begin{array}{ccc}
0&0&1/2\\
1&0&0\\
0&1&1/2\end{array}\right)&G_1&=\left(\begin{array}{ccc}
1&0&1/2\\
0&1&0\\
0&0&1/2\end{array}\right)
\end{align*}
Then for a general triangle partition sequence  given by $(\sigma, \tau_0,\tau_1)$ we simply define 
\begin{eqnarray*}
G_0(\sigma, \tau_0,\tau_1)  &  =  & \sigma G_0\tau_0 \\
G_1  (\sigma, \tau_0,\tau_1)  &  =  & \sigma G_1   \tau_0
\end{eqnarray*}

For notation, we say $\Gamma_{(\sigma, \tau_0,\tau_1) }  (i_1,\dots,i_n)$ corresponds to the triangle whose vertices are given by $$\pi(V G_{i_0}(\sigma, \tau_0,\tau_1)  \cdots    G_{i_n}(\sigma, \tau_0,\tau_1)).$$
  
 Let us look at a few examples.   Consider the $(e,e,e)$ case.   By direct calculation, we have 
 
 \begin{eqnarray*} 
 VG_0(e,e,e) G_0(e,e,e) &=&  \left(\begin{array}{ccc} 1&1&1\\ 0&1&1\\0&0&1\end{array}\right)
   \left(\begin{array}{ccc} 0&0&1/2\\ 1&0&0\\0&1&1/2\end{array}\right)  \left(\begin{array}{ccc} 0&0&1/2\\ 1&0&0\\0&1&1/2\end{array}\right)  \\
   &=&  \left(\begin{array}{ccc} 1 &1  &1/2\\ 1&    1/2  &    3/4   \\1&   1/2  &1/4   \end{array}\right) \\
    VG_0(e,e,e) G_1  (e,e,e)  &=&   \left(\begin{array}{ccc} 1 &1  &1    \\ 1&    1  &    3/4   \\0    &   1  &1/4   \end{array}\right) \\
    VG_1(e,e,e) G_0  (e,e,e)  &=&   \left(\begin{array}{ccc} 1 &1  &1    \\ 0   &    1/2 &    1/4   \\0    &   1/2  &1/4   \end{array}\right)  \\
    VG_1(e,e,e) G_1  (e,e,e)  &=&   \left(\begin{array}{ccc} 1 &1  &1    \\ 0   &    1 &    1/4   \\0    &   0  &1/4   \end{array}\right) 
 \end{eqnarray*}

  Pictorially, we have
      
     
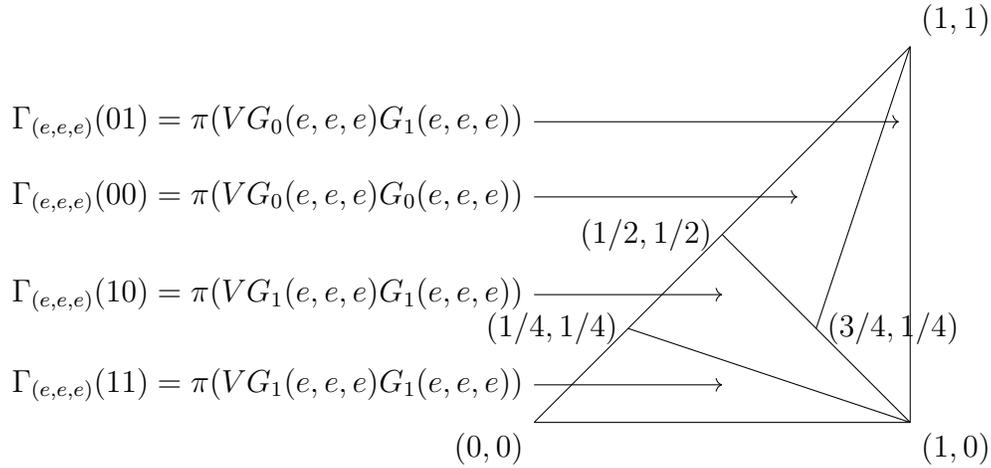
\begin{center}  
\begin{tikzpicture}[scale=5]
\draw(0,0)--(1,0);
\draw(0,0)--(1,1);
\draw(1,0)--(1,1);
\draw(1,0)--(1/2,1/2);
\draw(3/4,1/4)--(1,1);
\draw(1,0)--(1/4,1/4);

\draw[->](0,.8)node[left]   {$\Gamma_{(e,e,e)}(01)= \pi(VG_0(e,e,e)G_1(e,e,e))$  }--(.97,.8);

\draw[->](0,.6)node[left] {$\Gamma_{(e,e,e)}(00)= \pi(VG_0(e,e,e)G_0(e,e,e))$  }--(.7,.6);

\draw[->](0,.34)node[left]     {$\Gamma_{(e,e,e)}(10)= \pi(VG_1(e,e,e)G_1(e,e,e))$  }    --(.5,.34);

\draw[->](0,.10)node[left]      {$\Gamma_{(e,e,e)}(11)= \pi(VG_1(e,e,e)G_1(e,e,e))$  }--(.5,.10);

\node[below left]at(0,0){$(0,0)$};
\node[below right]at(1,0){$(1,0)$};
\node[above right]at(1,1){$(1,1)$};
\node[left]at(1/2,1/2){$(1/2,1/2)$};
\node[right]at(3/4,1/4){$(3/4,1/4)$};
\node[left]at(1/4,1/4){$(1/4,1/4)$};
\end{tikzpicture}
\end{center}

Note that this partitioning looks similar, but is different, from the corresponding partitioning $\triangle_{(e,e,e)}(00)   $, $\triangle_{(e,e,e)}(01)   $, $\triangle_{(e,e,e)}(10)   $ and
$\triangle_{(e,e,e)}(11)   $,

  We could also do similar calculations for the $((12),(13), e))$ case, getting

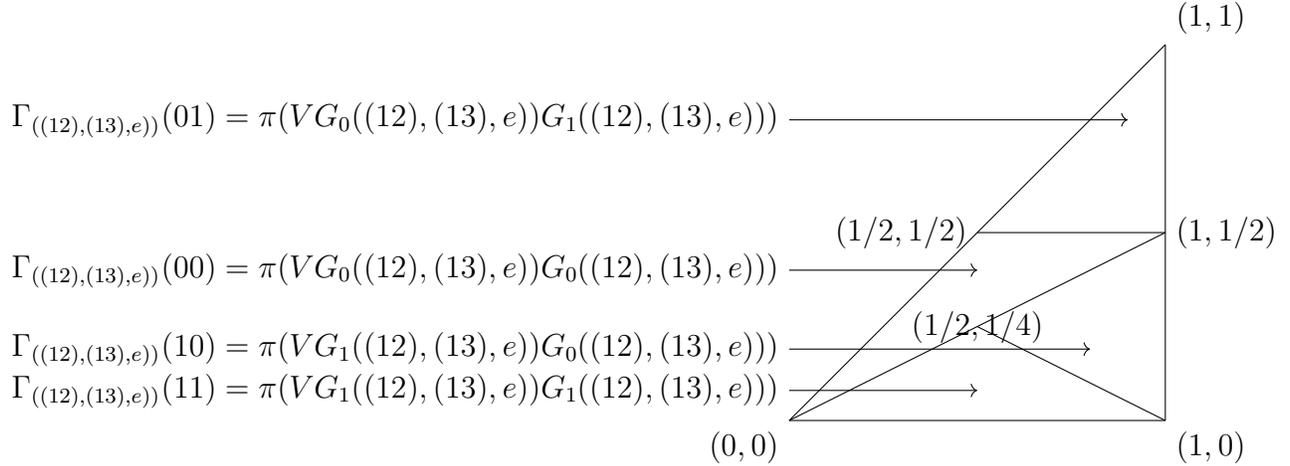
\begin{center}  
\begin{tikzpicture}[scale=5]
\draw(0,0)--(1,0);
\draw(0,0)--(1,1);
\draw(1,0)--(1,1);
\draw(0,0)--(1,1/2);
\draw(1/2,1/2)--(1,1/2);
\draw(1/2,1/4)--(1,0);

\draw[->](0,.8)node[left] {$\Gamma_{((12),(13), e))}(01)= \pi(VG_0((12),(13), e))G_1((12),(13), e)))$  }--(.9,.8);
\draw[->](0,.4)node[left] {$\Gamma_{((12),(13), e))}(00)= \pi(VG_0((12),(13), e))G_0((12),(13), e)))$  }--(.5,.4);

\draw[->](0,.19)node[left] {$\Gamma_{((12),(13), e))}(10)= \pi(VG_1((12),(13), e))G_0((12),(13), e)))$  }   --(4/5,.19);

\draw[->](0,.08)node[left]    {$\Gamma_{((12),(13), e))}(11)= \pi(VG_1((12),(13), e))G_1((12),(13), e)))$  }--(1/2,.08);

\node[below left]at(0,0){$(0,0)$};
\node[below right]at(1,0){$(1,0)$};
\node[above right]at(1,1){$(1,1)$};
\node[right]at(1,1/2){$(1,1/2)$};
\node[left]at(1/2,1/2){$(1/2,1/2)$};
\node at(1/2,1/4){$(1/2,1/4)$};
\end{tikzpicture}
\end{center}

Note again that  this partitioning looks similar, but is different, from the corresponding partitioning $\triangle_{((12),(13), e))}(00)   $, $\triangle_{((12),(13), e))}(01)   $, $\triangle_{((12),(13), e))}(10)   $ and
$\triangle_{((12),(13), e))}(11)   $.

Each choice of a triple $(\sigma, \tau_0, \tau_1) \in S_3$ will given rise to finer and finer partitionings by triangles 
$\Gamma_{( \sigma, \tau_0,  \tau_1)}(i_0. i_1 ,\ldots ,i_n)$.

\begin{definition}
For any $    (\sigma, \tau_0, \tau_1)     \in S_3^3$, we define the $n$th Barycentric partition $\mathcal{B}_n (\sigma, \tau_0, \tau_1) $  of the triangle $\triangle$ as the partitioning  given by the subtriangles $\Gamma_{(\sigma, \tau_0, \tau_1)}(i_1, i_2, i_3, \ldots , i_n) $.

\end{definition} 
We will denote the limit partition as $\mathcal{B}( \sigma. \tau_0, \tau_1)$.

This leads to another natural way to classify elements in the triangle
$\triangle$, namely

\begin{definition} A element $(\alpha, \beta) \in \triangle$ has Barycentric-$ (\sigma, \tau_0, \tau_1) $ TRIP sequence $(i_1, i_2, i_3, \ldots )$, where each $i_j$ is either zero or one, if, for all $n$,
$$(\alpha, \beta) \in \Gamma_{(\sigma, \tau_0, \tau_ 1 )}(i_1, i_2, i_3, \ldots, i_n ).$$
\end{definition}

We have seen that periodicity of Farey-$ (\sigma, \tau_0, \tau_1) $ TRIP sequences are linked to the possibility of the pair $(\alpha, \beta)$ being in the same cubic number field. For Barycentric-$ (\sigma, \tau_0, \tau_1) $ triangle sequences, periodicity is linked to  $(\alpha, \beta)$ being rationals.

\begin{theorem} \label{rational} Suppose that $(\alpha, \beta) \in \triangle$ has an eventually periodic Barycentric-$(\sigma, \tau_0, \tau_1) $ TRIP sequence $(i_0. i_1, i_2, \ldots )$ and that the nested sequence of triangles $\Gamma_{(\sigma, \tau_0, \tau_1)}  (   i_0  ) \supset \triangle_{(\sigma, \tau_0, \tau_1)}  (   i_0 i_1 )   \supset \triangle_{(\sigma, \tau_0, \tau_1)}  (   i_0 i_1 i_2  ) \supset \cdots  $   converge to a point. Then $\alpha$ and $\beta$ are  both rational numbers.
\end{theorem}

The proof is almost exactly the same as the proof in Theorem 6.2  in \cite{tripmaps}.  That proof shows that there  is an invertible $3 \times 3$ matrix $A$ with integer entries so that 
$$A \left(  \begin{array}{c} 1 \\ \alpha \\ \beta \end{array} \right)$$
  is an eigenvector of a finite product of matrices $G_0(\sigma, \tau_0,\tau_1)$ and $G_1(\sigma, \tau_0,\tau_1).$  Unlike the two matrices $F_0(\sigma, \tau_0,\tau_1)$  and $F_1(\sigma, \tau_0,\tau_1)$, both of the matrices $G_0(\sigma, \tau_0,\tau_1)$ and $G_1(\sigma, \tau_0,\tau_1)$ are Markov (meaning that the sums of the columns is always one).  It is well-known that the eigenvectors of Markov matrices must be rational.

As with the  TRIP  sequences, as we mentioned  in subsection \ref{multiplicative}, what we are doing here  is  the additive approach. Of course, there is also a Barycentric multiplicative version.  We simply set 
  $\Gamma_k   ( \sigma. \tau_0, \tau_1)   = \pi(VF_1^k( \sigma. \tau_0, \tau_1) F_0( \sigma. \tau_0, \tau_1) $,  where $k$ is a non-negative integer.   Then the map
$$T^B( \sigma. \tau_0, \tau_1):\triangle \rightarrow \triangle,   $$
for an $(x,y)$ in the interior of $     \Gamma_k   ( \sigma. \tau_0, \tau_1) $, is
$$T^B( \sigma. \tau_0, \tau_1)(x,y) = \pi \left(    V(VG_1 (\sigma , \tau_0, \tau_1 ) ^kG_0   (\sigma , \tau_0, \tau_1 )  )^{-1} \left(   \begin{array}{c} 1\\ x  \\ y \end{array}\right)   \right).$$

As before,  we can associate to each $(x,y) \in \triangle$ the sequence of non-negative integers $(a_1,a_2, \dots )$ if 
$$(x,y) \in  \Gamma_{a_1}   ( \sigma. \tau_0, \tau_1), T^B(x,y) \in  \Gamma_{a_2}   ( \sigma. \tau_0, \tau_1), T^B(^BT^B(x,y)) \in  \Gamma_{a_3}   ( \sigma. \tau_0, \tau_1), \ldots $$
(We are again are suppressing here the $ ( \sigma. \tau_0, \tau_1)$.)
We can still in straightforward way translate between the multiplicative version and the additive version, as we simply take any non-negative $k$ and replace it by $k$ ones followed by a single $0$.

  \section{The Analog to the Minkowksi Question Mark Function} \label{analog}
 
  For each choice of $(\sigma,\tau_0,\tau_1) \in S_3 \times S_3 \times S_3,$ we have two natural partitionings of the triangle $\triangle$, namely the Farey partition $\mathcal{F}$ and the Barycentric partition $\mathcal{B}$.

   For each sequence $(i_1,\dots,i_n)\in\{0,1\}^n$, define
   
   \[\Phi_n  (\sigma,\tau_0,\tau_1) :\triangle\to\triangle\]
by $\Phi_n(\triangle_{(\sigma,\tau_0,\tau_1)}(i_1,\dots,i_n))=\Gamma_{(\sigma,\tau_0,\tau_1)}(i_1,\dots,i_n)$ for every $(i_1,\dots,i_n)\in\{0,1\}^n$ by mapping vertices to corresponding vertices and extending to interiors via linearity.  

Let us look at an example.  For the triple $(e,e,e)$, we earlier saw that

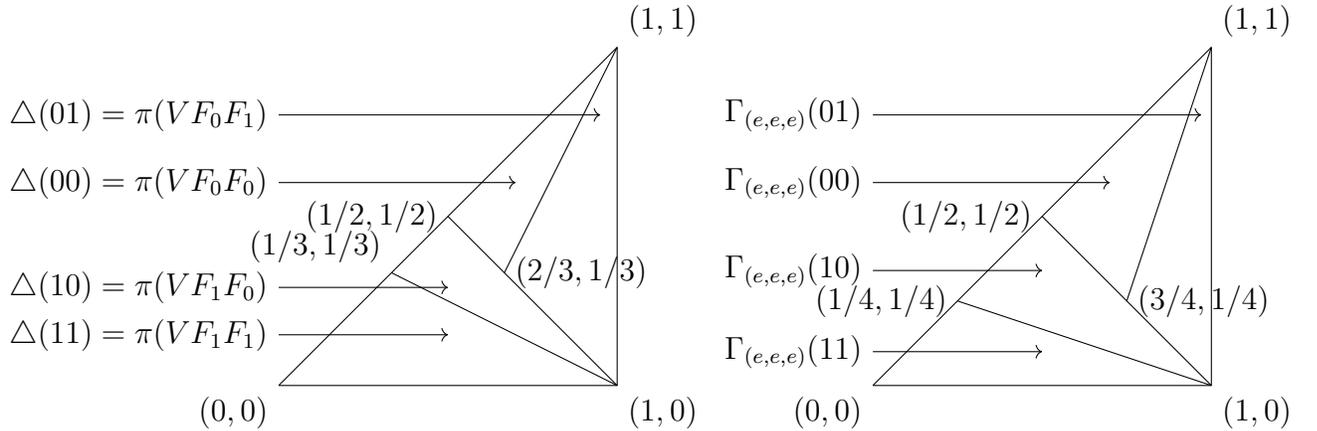
\begin{tikzpicture}[scale=4.5]
\draw(0,0)--(1,0);
\draw(0,0)--(1,1);
\draw(1,0)--(1,1);
\draw(1,0)--(1/2,1/2);
\draw(2/3,1/3)--(1,1);
\draw(1,0)--(1/3,1/3);

\draw[->](0,.8)node[left]{$\triangle(01)= \pi(VF_0F_1)$}--(.95,.8);

\draw[->](0,.6)node[left]{$\triangle(00)= \pi(VF_0F_0)$}--(.7,.6);

\draw[->](0,.29)node[left]{$\triangle(10)= \pi(VF_1F_0)$}--(.5,.29);

\draw[->](0,.15)node[left]{$\triangle(11) = \pi(VF_1F_1)$}--(.5,.15);

\node[below left]at(0,0){$(0,0)$};
\node[below right]at(1,0){$(1,0)$};
\node[above right]at(1,1){$(1,1)$};
\node[left]at(1/2,1/2){$(1/2,1/2)$};
\node[right]at(2/3,1/3){$(2/3,1/3)$};
\node[above left]at(1/3,1/3){$(1/3,1/3)$};
\end{tikzpicture}    \begin{tikzpicture}[scale=4.5]
\draw(0,0)--(1,0);
\draw(0,0)--(1,1);
\draw(1,0)--(1,1);
\draw(1,0)--(1/2,1/2);
\draw(3/4,1/4)--(1,1);
\draw(1,0)--(1/4,1/4);

\draw[->](0,.8)node[left]   {$\Gamma_{(e,e,e)}(01)$  }--(.97,.8);

\draw[->](0,.6)node[left] {$\Gamma_{(e,e,e)}(00)$  }--(.7,.6);

\draw[->](0,.34)node[left]     {$\Gamma_{(e,e,e)}(10)$  }    --(.5,.34);

\draw[->](0,.10)node[left]      {$\Gamma_{(e,e,e)}(11)$  }--(.5,.10);

\node[below left]at(0,0){$(0,0)$};
\node[below right]at(1,0){$(1,0)$};
\node[above right]at(1,1){$(1,1)$};
\node[left]at(1/2,1/2){$(1/2,1/2)$};
\node[right]at(3/4,1/4){$(3/4,1/4)$};
\node[left]at(1/4,1/4){$(1/4,1/4)$};
\end{tikzpicture}

Then we have 
\begin{eqnarray*}
\Phi_2( e, e , e )  ( 0,0) & = & (0, 0)  \\
\Phi_2( e, e , e ) (1,0) &=& (1,0)  \\
\Phi_2( e, e , e ) (1/3, 1/3) &=& (1/3, 1/3)   \\
\Phi_2( e, e , e )  (2/3, 1/3) &=& (3/4, 1/4)  \\
\Phi_2( e, e , e ) (1/2, 1/2) &=& (1/2, 1/2)  \\
\Phi_2( e, e , e ) (1,1) &=& (1,1).
\end{eqnarray*}

We would like to define our new function to be
\[\Phi(\sigma,\tau_0,\tau_1) =\lim_{n\to\infty}\Phi_n (\sigma,\tau_0,\tau_1),\]
but we must exercise care.  The difficulty is that there are times,  for various  choices of $(\sigma,\tau_0,\tau_1)\in S_3$, that  we do not have convergence of the $\triangle_{(\sigma,\tau_0,\tau_1)}(i_1,\dots,i_n)$ or for  the $\Gamma_{(\sigma,\tau_0,\tau_1)} (i_1,\dots,i_n)$, or, in other words, it is not necessarily true, for a fixed sequence $(i_1, i_2, i_3, \ldots )$ that 

$$\lim_{n\rightarrow \infty} \triangle_{ (\sigma,\tau_0,\tau_1)}(i_1,\dots,i_n) = \;\mbox{single point} \; \mbox{or} \;   \lim_{n\rightarrow \infty}\Gamma_{ (\sigma,\tau_0,\tau_1))}(i_1,\dots,i_n) = \;\mbox{single point}   $$

We do know that these intersections  are either single points or line segments.  Thus we need our $\Phi  (\sigma,\tau_0,\tau_1)$ not to be a functions] sending points to points but to be a relation, sending a point to either a single point or a segment.  

We are now ready to officially define the $(\sigma, \tau_0, \tau_1)$ question mark function $\Phi  (\sigma,\tau_0,\tau_1)$.  Given an $(x,y)\in \triangle$, there are two possibilities. If there is a sequence 
$(i_1,\dots,i_n)$ with $(x,y) \in \partial (\triangle_{(\sigma,\tau_0,\tau_1)}(i_1,\dots,i_n))$, then $\Phi(\sigma,\tau_0,\tau_1)(x,y)$ will be the corresponding point, via linearity, on  $\partial (\Gamma_{(\sigma,\tau_0,\tau_1)}(i_1,\dots,i_n))$.   This only happens on a set of measure zero.  Thus for almost all $(x,y)\in \triangle$, there is an infinite sequence $(i_1, \i_2, i_3, \dots )$ with $(x,y)$ in the interior of each $\triangle_{(\sigma,\tau_0,\tau_1)}(i_1,\dots,i_n))$.  

\begin{definition} For a given $(\sigma,\tau_0,\tau_1)$ and for an  $(x,y)\in \triangle$, suppose that there is an infinite sequence $(i_1, \i_2, i_3, \dots )$ with $(x,y)$ in the interior of each $\triangle_{(\sigma,\tau_0,\tau_1)}(i_1,\dots,i_n))$.  
Then  $\Phi(\sigma,\tau_0,\tau_1)(x,y)$ is  the following line segment or point:
$$\Phi(\sigma,\tau_0,\tau_1)(x,y) =    \lim_{n\rightarrow \infty} \Gamma_{(\sigma,\tau_0,\tau_1)}(i_1,\dots,i_n) .$$
\end{definition}

Actually, for some $(\sigma, \tau_0, \tau_1)$, we have that $ \lim_{n\rightarrow \infty} \Gamma_{(\sigma,\tau_0,\tau_1)}(i_1,\dots,i_n) $ will be a single point, in which case of course, $\Phi(\sigma,\tau_0,\tau_1)(x,y) $ will be an actual function.  Further, for some $(\sigma, \tau_0, \tau_1)$, we have that $ \lim_{n\rightarrow \infty} \Gamma_{(\sigma,\tau_0,\tau_1)}(i_1,\dots,i_n) $ is almost everywhere a single point, in whch case $\Phi(\sigma,\tau_0,\tau_1)(x,y) $ will be an actual function almost everywhere.  But there are $(\sigma, \tau_0, \tau_1)$ for which $ \lim_{n\rightarrow \infty} \Gamma_{(\sigma,\tau_0,\tau_1)}(i_1,\dots,i_n) $ is never a point.  We classify these different types in section \ref{convergence}.

\section{Cubics to Rationals}
The classical Minkowski question mark function sends quadratic irrationals to rational numbers.  We have the following partial analog:

\begin{theorem} For a given $(\sigma, \tau_0, \tau_1)$ and a sequence $(i_1, i_2, i_3, \ldots )$ of zeros and ones,  suppose that 
\begin{eqnarray*}
\lim_{n\rightarrow \infty} \triangle_{ (\sigma,\tau_0,\tau_1)}(i_1,\dots,i_n) &=& (\alpha, \beta) \\
  \lim_{n\rightarrow \infty}\Gamma_{ (\sigma,\tau_0,\tau_1))}(i_1,\dots,i_n) &=&(r,s).
\end{eqnarray*}
Then $(\alpha, \beta)$ are algebraic numbers in the same number field of degree less than or equal to three, $(r,s)$ are rational numbers and
$$\Phi(\sigma,\tau_0,\tau_1)(\alpha, \beta) = (r,s).$$
\end{theorem}
This follows immediately from Theorem \ref{cubic} and Theorem \ref{rational}.

\section{Definition of Singularness} \label{definition}
Since our analogs of the Minkowski Question Mark function need not be functions but instead could be relations, we need to exercise some care in the definition of singularity.

\begin{definition}\label{trad} A  relation $f:\triangle\to\triangle$ is singular with respect to the Lebesgue measure $\lambda$ if there exists a set $A$ in the domain of $f$ such that $\lambda(A)=\lambda(\triangle)$ while $\lambda(f(A))=0$.
\end{definition}

\section{Reduction of Cases} \label{reduction}
From the above we see that for each of the 216 different TRIP maps there is only one reasonable candidate for an analog for the Minkowski question mark function.  Thus in principle, we could examine each of these 216 different maps and see which ones are singular.  Of course, this hardly seems worth the effort.  Luckily we can  place these 216 maps into 15 classes whose associated generalization of the question mark function is related by a linear transformation and show for five of these classes that this function is singular.

\subsection{Twins}
The first reduction comes from the fact that no $(\sigma, \tau_0, \tau_1)$   uniquely partitions $\triangle$.  Instead, every TRIP sequence  has a "twin" that is essentially just the same with matrices reversed.    Then the $\Phi$ defined from a TRIP map will be the same as the $\Phi$ defined from it's twin.

\begin{lemma}
\label{pairing} 
$F_{\sigma,\tau_0,\tau_1}$ and $F_{\sigma(13),(12)\tau_1,(12)\tau_0}$ give the same partition of $\triangle$.
\end{lemma}
\begin{proof}
Recall our matrices $F_0$ and $F_1$.
Then we calculate $(13)F_0(12)$ and $(13)F_1(12)$:
\begin{align*}
(13)F_0(12)=\left(\begin{array}{ccc}
0&0&1\\
0&1&0\\
1&0&0\end{array}\right)\left(\begin{array}{ccc}
0&0&1\\
1&0&0\\
0&1&1\end{array}\right)\left(\begin{array}{ccc}
0&1&0\\
1&0&0\\
0&0&1\end{array}\right)&=\left(\begin{array}{ccc}
1&0&1\\
0&1&0\\
0&0&1\end{array}\right)=F_1\\
(13)F_1(12)=\left(\begin{array}{ccc}
0&0&1\\
0&1&0\\
1&0&0\end{array}\right)\left(\begin{array}{ccc}
1&0&1\\
0&1&0\\
0&0&1\end{array}\right)\left(\begin{array}{ccc}
0&1&0\\
1&0&0\\
0&0&1\end{array}\right)&=\left(\begin{array}{ccc}
0&0&1\\
1&0&0\\
0&1&1\end{array}\right)=F_0
\end{align*}
Thus the Farey partition $\mathcal{F}(\sigma,\tau_0,\tau_1 )$ is the same as the Farey partition for $\mathcal{F}( \sigma (13),(12) \tau_1,(12) \tau_0)$.  

For example, we would have 
$$\triangle_{  (\sigma,\tau_0,\tau_1 ) } ( 1,1,0,0,0,1) = \triangle_{ ( \sigma (13) , (12) \tau_1,(12)\tau_0)} (0,0,1,1,1,0).$$
\end{proof}
The twin phenomenon comes down to what we could have labeled each  $F_0$ matrix as the $F_1$ matrix and each $F_1$ matrix as $F_0$, and get the same partitioning. 

With this pairing, this reduces the number of cases we need to check to 108.

\subsection{Reduction  to $21$ cases}
We can further reduce the number of cases by grouping permutations into classes whose partitions of $\triangle$ are related by a linear transformation.

Recall that our original triangle was represented by the vectors $(1,0,0), (1,1,0),$ and $(1,1,1)$ in $\R^3$.  These define a cone in $\R^3$, and each TRIP sequence provides a systematic way of partitioning this cone into smaller and smaller subcones through linear combinations of these initial vectors.  We translate this cone partition into a triangle partition by projecting onto the plane $x_1=1$.  We can similarly think of our barycentric partitioning in the same way.  However, our initial choice of cone shouldn't matter; we should be able to define our cone using any three linearly independent vectors in $\Z^3$ and define our projection map onto the plane defined by the endpoints of these vectors.  Because moving from one cone and projection map to another can be described by a linear transformation, this means that if $\Phi    ( \sigma,\tau_0, \tau_1)$ is singular for the original multidimensional fraction defined on $\triangle$ by the matrix $V$, it will be singular on any cone defined by the matrix $MV$ where $M$ is the linear transformation the sends our original cone to the new one.

\begin{theorem} 
\label{ordering}
If the permutation $(\sigma,\tau_0,\tau_1)$ gives rise to a singular function from $\triangle\to\triangle$, then, for any triangle $\triangle'$ with vertices given in projective coordinates by the matrix $V'=(v_1, v_2, v_3) $,  this permutation gives rise to a singular function from $\triangle'\to\triangle'$ under the same definition.
\end{theorem}

\begin{proof}
The proof is straightforward linear algebra.
\end{proof}

This gives us the following result.
\begin{corollary} 
\label{permute}
If the $\Phi(\sigma,\tau_0,\tau_1)$  is singular, then so is the $\Phi (\rho\sigma, \tau_0 \rho^{-1},\tau_1 \rho^{-1})$, for any $\rho \in S_3.$
\end{corollary}
\begin{proof}
Note that these two partitions only different in the ordering of the initial vertices of $\triangle$.  Then this result follows directly from Theorem \ref{ordering}.
\end{proof}

By this corollary, we can greatly reduce the number of cases that we need to check.  We will informally define a class of permutations to be those related by a combination of Lemma \ref{pairing} or Corollary \ref{permute}.  By Corollary \ref{permute}, we know that there will always be a permutation of the form $(e,\tau_0,\tau_1)$ in each class, which means there are at most 36 classes to check.  We would expect that by Lemma \ref{pairing} we could reduce this number to 18.  However, we find that in several cases, the twin of $(e,\tau_0,\tau_1)$ is on the list of the six permutation triples given by \ref{permute}, so instead we find 21 classes, represented by the following permutations:
\begin{align*}
&(e,e,e)&&(e,e,12)&&(e,e,13)&&(e,e,23)&&(e,e,123)*&&(e,e,132)&&(e,12,e)\\
&(e,12,12)&&(e,12,13)*&&(e,12,23)&&(e,12,132)&&(e,13,e)&&(e,13,12)*&&(e,13,23)\\
&(e,13,132)&&(e,23,e)&&(e,23,23)*&&(e,23,132)&&(e,123,e)*&&(e,123,132)&&(e,132,132)*\\
\end{align*}
The starred permutations represent classes with 6 maps, while the unstarred ones represent classes with 12 maps. 

\subsection{Reduction to 15 Cases}
We now reduce the above 21 cases to just 15.  

We present the following lemma from \cite{tripmaps}

\begin{lemma}\label{monk1} The TRIP maps  $F_{e,(23),(23)},$ $F_{e,(23),(132)},$ $F_{e,(132),(23)},$ $F_{e,(132),(132)},$ $F_{(13),(132),(132)},$ $F_{(13),(23),(132)},$ $F_{(13),(132),(23)},$ and $F_{(13),(23),(23)}$ all give the same partition of $\triangle$.
\end{lemma}
\begin{proof} The permutation $(e,23,23)$ gives matrices:
\begin{align*}
F_0&=\left(\begin{array}{ccc}
0&1&0\\
1&0&0\\
0&1&1\end{array}\right)&F_1&=\left(\begin{array}{ccc}
1&1&0\\
0&0&1\\
0&1&0\end{array}\right)
\end{align*}
The permutation $(e,23,132)$ gives matrices:
\begin{align*}
F_0&=\left(\begin{array}{ccc}
0&1&0\\
1&0&0\\
0&1&1\end{array}\right)&F_1&=\left(\begin{array}{ccc}
0&1&1\\
1&0&0\\
0&1&0\end{array}\right)
\end{align*}
The permutation $(e,132,23)$ gives matrices:
\begin{align*}
F_0&=\left(\begin{array}{ccc}
0&1&0\\
0&0&1\\
1&1&0\end{array}\right)&F_1&=\left(\begin{array}{ccc}
1&1&0\\
0&0&1\\
0&1&0\end{array}\right)
\end{align*}
The permutation $(e,132,132)$ gives matrices:
\begin{align*}
F_0&=\left(\begin{array}{ccc}
0&1&0\\
0&0&1\\
1&1&0\end{array}\right)&F_1&=\left(\begin{array}{ccc}
0&1&1\\
1&0&0\\
0&1&0\end{array}\right)
\end{align*}
For the base case, we note that the new vertex of $\triangle(i_1)$ is on the edge between $(1,0,0)$ and $(1,1,0)$ for each map. For the inductive step $v_{k-1,1}$ is the unique vertex of $\triangle_{(\sigma,\tau_0,\tau_1)}(i_1,\dots,i_{k-1})$ which is not a vertex of $\triangle_{(\sigma,\tau_0,\tau_1)}(i_1,i_2,\dots,i_{k-2})$.  Because the new vertex of $\triangle_{(\sigma,\tau_0,\tau_1)}(i_1,i_2,\dots,i_k)$ is $v_{k-1,2}+v_{k-1,3}$ for either choice of $i_k$, and this is the sum of two vertices that were not new at the previous step, each map gives the same partition.  Thus each of these four maps gives the same partition of $\triangle$.  The remaining four maps come from applying Lemma \ref{pairing} to these four maps.
\end{proof}

This means that proving singularity for any of these eight classes  proves singularity for the remaining seven. The list of 21 class from the previous subsection includes three of these classes: $(e, 23, 23), (e, 23, 132)$ and $(e, 132, 132)$.  Thus we can delete two of them, reducing our work to checking 19 classes.  

 Further, from \cite{tripmaps} we know that the M\"{o}nkemeyer map is the same as $T_{(e,132,23)}$ in two dimensions and it is the map used by Panti in \cite{panti} to derive his generalization of the Minkowski function, so we will call the family of maps represented by $(e,23,132)$ "M\'{o}nkemeyer-type maps."

This next lemma deal with maps that partition $\triangle$ into subtriangles that all contain the vertex $(1,0)$ of $\triangle$, so in essence what they do is provide a Farey partitioning of the hypotenuse of $\triangle$. We will use the notation that $\triangle_{n,F}$ refers to the set of all $\triangle(i_1,i_2,\dots,i_n)$ for a given TRIP map.

\begin{lemma} The TRIP maps $F_{e,(12),e},$ $F_{e,(123),e},$ $F_{e,(12),(13)},$ $F_{e,(123),(13)},$ $F_{(13),(12),e},$ $F_{(13),(12),(13)},$ $F_{(13),(123),e},$ and $F_{(13),(123),(13)}$ all give the same partition of $\triangle$.
\end{lemma}
\begin{proof} The permutation $(e,12,e)$ gives matrices:
\begin{align*}
F_0&=\left(\begin{array}{ccc}
0&0&1\\
0&1&0\\
1&0&1\end{array}\right)&F_1&=\left(\begin{array}{ccc}
1&0&1\\
0&1&0\\
0&0&1\end{array}\right)
\end{align*}
The permutation $(e,123,e)$ gives matrices:
\begin{align*}
F_0&=\left(\begin{array}{ccc}
1&0&0\\
0&1&0\\
1&0&1\end{array}\right)&F_1&=\left(\begin{array}{ccc}
1&0&1\\
0&1&0\\
0&0&1\end{array}\right)
\end{align*}
The permutation $(e,12,13)$ gives matrices:
\begin{align*}
F_0&=\left(\begin{array}{ccc}
0&0&1\\
0&1&0\\
1&0&1\end{array}\right)&F_1&=\left(\begin{array}{ccc}
1&0&1\\
0&1&0\\
1&0&0\end{array}\right)
\end{align*}
The permutation $(e,123,13)$ gives matrices:
\begin{align*}
F_0&=\left(\begin{array}{ccc}
1&0&0\\
0&1&0\\
1&0&1\end{array}\right)&F_1&=\left(\begin{array}{ccc}
1&0&1\\
0&1&0\\
1&0&0\end{array}\right)
\end{align*}
Note that all these maps fix the vertex $(1,1,0)$ of $\triangle$.  Then $\triangle_{1,F}$ is the same for all these maps.  Now suppose that $\triangle_{n,F}$ is the same for all these maps.  Then for a given $T\in\triangle_{n,F}$ we have that $\{T(0),T(1)\}$, where $T(i)$ is obtained to applying $F_i$ to $T$, is the same for all these maps so we get that $\triangle_{n+1,F}$ is the same for all these maps.  The remaining four maps come from applying Lemma \ref{pairing} to these four maps.
\end{proof}

Of course this means that proving singularity for any of these eight classes  proves singularity for the remaining seven. The list of 21 class from the previous subsection includes three of these classes: $(e,12,e), (e,123,e)$ or $(e,12,13)$.  Thus we can reduce our list of classes to check to 17 classes.

 We will call the family represented by $(e,12,e)$ "degenerate Farey maps" because the maps in this family simply give a Farey partioning of one of the sides of $\triangle$.  

\begin{lemma} The TRIP maps $F_{e,e,(12)}$, $F_{e,e,(123)}$, $F_{e,(13),(12)}$, $F_{e,(13),(123)}$, $F_{(13),e,(12)}$, $F_{(13),(13),(12)}$, $F_{(13),e,(123)}$ and $F_{(13),(13),(123)}$ give the same partition of $\triangle$.
\end{lemma}
\begin{proof}
The permutation $(e,e,12)$ gives the following matrices:
\begin{align*}
F_0&=\left(\begin{array}{ccc}
0&0&1\\
1&0&0\\
0&1&1\end{array}\right)&F_1&=\left(\begin{array}{ccc}
0&1&1\\
1&0&0\\
0&0&1\end{array}\right)
\end{align*}
The permutation $(e,e,123)$ gives the following matrices:
\begin{align*}
F_0&=\left(\begin{array}{ccc}
0&0&1\\
1&0&0\\
0&1&1\end{array}\right)&F_1&=\left(\begin{array}{ccc}
1&1&0\\
0&0&1\\
1&0&0\end{array}\right)
\end{align*}
The permutation $(e,13,12)$ gives the following matrices:
\begin{align*}
F_0&=\left(\begin{array}{ccc}
1&0&0\\
0&0&1\\
1&1&0\end{array}\right)&F_1&=\left(\begin{array}{ccc}
0&1&1\\
1&0&0\\
0&0&1\end{array}\right)
\end{align*}
The permutation $(e,13,123)$ gives the following matrices:
\begin{align*}
F_0&=\left(\begin{array}{ccc}
1&0&0\\
0&0&1\\
1&1&0\end{array}\right)&F_1&=\left(\begin{array}{ccc}
1&1&0\\
0&0&1\\
1&0&0\end{array}\right)
\end{align*}
Notice that all three permutations give the same initial partition of $\triangle$ and that $v_2(1)$ is the same in both triangles.  Now suppose that $\triangle_{n,F}$ is the same and that for corresponding triangles $v_2(n)$ is the same.  Then because the new vertex is $v_1(n)\hat{+}v_3(n)$ and the only difference between partitions is which vertex is $v_1(n)$ and which vertex is $v_3(n)$, we get that $\triangle_{n+1,F}$ is the same for these four maps.  The remaining four maps come from applying \ref{pairing} to these four maps.
\end{proof}
In the same way as we did before, we now know that proving singularity for any element in any of these classes immediately proves singularity for all elements in the other seven classes.  As the classes
 $(e,e,12)$, $(e,e,123)$ and $(e,13,12)$ are all in our original list of 21 classes, we can reduce our list by two more, resulting in 15 classes.  
 
Thus we have reduced the number of classes from 21 to the following 15:
\begin{align*}
&(e,e,e)&&(e,e,12)&&(e,e,13)&&(e,e,23)&&(e,e,132)\\
&(e,12,e)&&(e,12,12)&&(e,12,23)&&(e,12,132)&&(e,13,e)\\
&(e,13,23)&&(e,13,132)&&(e,23,e)&&(e,23,132)&&(e,123,132)
\end{align*}

\section{On when $\Phi$ is a function: Convergence on the Barycentric Side}\label{convergence}

As we have reduced our problem to investigating fifteen different classes of triangle partition maps, we can return to the question of when is $\Phi(\sigma, \tau_0, \tau_1)$ an actual function, as opposed to being a relation.  Recall from section   \ref{analog}, we will have a function if for all sequences $(i_1, i_2, i_3, \ldots)$, we have that 
$$ \lim_{n\rightarrow \infty} \Gamma_{(\sigma,\tau_0,\tau_1)}(i_1,\dots,i_n)  = \; \mbox{single point}.$$

We first will set notation.  Set $v_1(n), v_2(n)$ and $v_3(n)$  to be the  vertices of $\Gamma_{(\sigma,\tau_0,\tau_1)}(i_1,,\dots,i_n)$ whose ordering is given by $VG_{i_1}(\sigma,\tau_0,\tau_1)\cdots G_{i_n}(\sigma,\tau_0,\tau_1)$.  Let $\tau_n$ denote the length of the side from $v_1(n)$ to $v_2(n)$, $\rho_n$ denote the length of the side from $v_2(n)$ to $v_3(n)$, and $\mu_n$ denote the length of the side from $v_1(n)$ to $v_3(n)$:
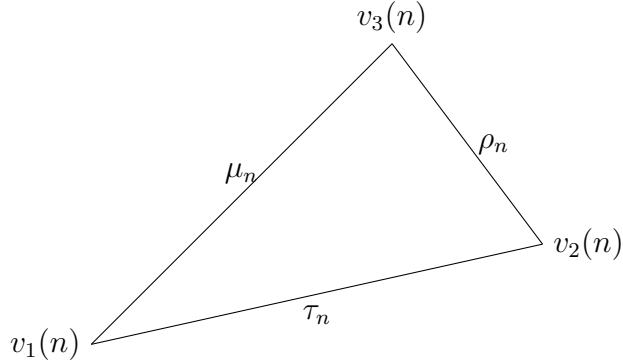
\begin{center}
\begin{tikzpicture}[scale=4]
\draw (0,0)--(1,1);
\draw (0,0)--(3/2,1/3);
\draw (3/2,1/3)--(1,1);
\node [left] at (0,0) {$v_1(n)$};
\node [right] at (3/2,1/3) {$v_2(n)$};
\node [above] at (1,1) {$v_3(n)$};
\node [below] at (3/4,1/6) {$\tau_n$};
\node [right] at (5/4,2/3) {$\rho_n$};
\node [above] at (1/2,1/2) {$\mu_n$};
\end{tikzpicture}
\end{center}

For a given  $(\sigma, \tau_0, \tau_1)$ and given sequence $(i_1, i_2, i_3,\ldots )$ of zero and ones, we will have  $ \lim_{n\rightarrow \infty} \Gamma_{(\sigma,\tau_0,\tau_1)}(i_1,\dots,i_n)  = \; \mbox{single point} $ when 
\[\lim_{n\to\infty}\tau_n=\lim_{n\to\infty}\rho_n=\lim_{n\to\infty}\mu_n  =0.\]

Then we have 
\begin{theorem}  For any triangle partition map $(\sigma,\tau_0,\tau_1)$  in  one of the classes $$ (e,e,23), (e,e,132), (e, 23, 132),$$
we have that $ \lim_{n\rightarrow \infty} \Gamma_{(\sigma,\tau_0,\tau_1)}(i_1,\dots,i_n)  = \; \mbox{single point},$
in which case $\Phi(\sigma, \tau_0, \tau_1)$ is a function.
\end{theorem}

\begin{proof}
We will show this just for the map $(e,e,23)$ as the others are similar.

We  have
\[G_0(e,e,23)=\left(\begin{array}{ccc}
0&0&1/2\\
1&0&0\\
0&1&1/2\end{array}\right)\]
This gives us 
\begin{align*}
\tau_n&=\rho_{n-1}\\
\rho_n&=\frac{1}{2}\mu_{n-1}\\
\mu_n&\leq\max(\tau_{n-1},\rho_{n-1})
\end{align*}
Then for any TRIP tree sequence with infinitely many zeros we have that \[\lim_{n\to\infty}\tau_n=\lim_{n\to\infty}\rho_n=\frac{1}{2}\lim_{n\to\infty}\mu_n\] which implies convergence to a point.

We also have 
\[G_1(e,e,23)=\left(\begin{array}{ccc}
1&1/2&0\\
0&0&1  \\
0&1/2&0\end{array}\right)\]
This gives us the following:
\begin{align*}
\tau_n&=\frac{1}{2}\mu_{n-1}\\
\rho_n&\leq\max(\tau_{n-1},\rho_{n-1})\\
\mu_n&=\tau_{n-1}
\end{align*}
Then whenever we have $\tau_1=(23)$ and a TRIP tree sequence with infinitely many ones we have 
\[\lim_{n\to\infty}\mu_n=\lim_{n\to\infty}\tau_n=\frac{1}{2}\lim_{n\to\infty}\mu_n\]
which implies that 
\[\lim_{n\to\infty}\mu_n=\lim_{n\to\infty}\tau_n=0\]
and hence convergence.  

Since there have to be an infinite number of zeros or an infinite number of ones (usually both) for any sequence $(i_1, i_2, i_3,\ldots ),$ we always have convergence.

\end{proof}

Now to turn to those for which we do not always have convergence.

\begin{theorem}  For any triangle partition map $(\sigma,\tau_0,\tau_1)$  in  one of the classes $$ (e,e,e), (e,e,12), (e, e, 13), (e, 23, e),$$
we have that $ \lim_{n\rightarrow \infty} \Gamma_{(\sigma,\tau_0,\tau_1)}(i_1,\dots,i_n)  = \; \mbox{single point}$
whenever the sequence $(i_1,i_2, i_3, \ldots )$ has an infinite number of zeros (i.e., when $(i_1,i_2, i_3, \ldots )\neq (i_1, \ldots i_k, \bar{1}$).
In these cases,  $\Phi(\sigma, \tau_0, \tau_1)$ is a function.
\end{theorem}

\begin{proof}
We prove this for the $(e,e,e)$ case, as the rest are similar. 

First, $G_0(e,e,e)= G_0(e,e,23)$, thus the argument in the previous proof will show that we have convergence whenever the sequence $(i_1,i_2, i_3, \ldots )$ has an infinite number of zeros.

What changes the situation from the previous theorem is the nature of $G_1(e,e,e).  $
We have
\[G_1(e,e,e)=\left(\begin{array}{ccc}
1&0&1/2\\
0&1&0\\
0&0&1/2\end{array}\right)\]
giving us
\begin{align*}
\tau_n&=\tau_{n-1}\\
\rho_n&\leq\max(\tau_{n-1},\rho_{n-1})\\
\mu_n&=\frac{1}{2}\mu_{n-1}
\end{align*}
which gives us no information about convergence to a point.

\end{proof}

\begin{theorem}  For any triangle partition map $(\sigma,\tau_0,\tau_1)$  in  one of the classes $$ (e,12,23), (e, 12, 132), (e, 13, 23), (e,13,132), (e,123,132),$$
we have that $ \lim_{n\rightarrow \infty} \Gamma_{(\sigma,\tau_0,\tau_1)}(i_1,\dots,i_n)  = \; \mbox{single point}$
whenever the sequence $(i_1,i_2, i_3, \ldots )$ has an infinite number of ones (i.e., when $(i_1,i_2, i_3, \ldots )\neq (i_1, \ldots i_k, \bar{0}$).
In these cases,  $\Phi(\sigma, \tau_0, \tau_1)$ is a function.
\end{theorem}

\begin{proof} We will show this for $ (e,12,23)$, as the others are similar.  First, $G_1(e, 12, 23) = G_1(e,e,23)$, meaning that we can use the arguments of the first proof in this subsection to see that we get convergence whenever the sequence $(i_1,i_2, i_3, \ldots )$ has an infinite number of ones.

We have
\[G_0(e,12,23)=\left(\begin{array}{ccc}
0&0&1/2\\
0&1&0\\
1&0&1/2\end{array}\right)\]
This gives us the following:
\begin{align*}
\tau_n&=\rho_{n-1}\\
\rho_n&\leq\max(\tau_{n-1},\rho_{n-1})\\
\mu_n&=\frac{1}{2}\mu_{n-1}
\end{align*}
Then whenever we have $\tau_0=(12)$ and a TRIP tree sequence with infinitely many zeros we have that 
\begin{align*}
\lim_{n\to\infty}\tau_n&=\lim_{n\to\infty}\rho_n\\
\lim_{n\to\infty}\mu_n&=0
\end{align*}
which does not tell us that we are guaranteed convergence  to a point. 

\end{proof}

\begin{theorem} For any triangle partition map $(\sigma,\tau_0,\tau_1)$  in one of the classes $$ (e,12,12), (e,13,e)$$
we have that $ \lim_{n\rightarrow \infty} \Gamma_{(\sigma,\tau_0,\tau_1)}(i_1,\dots,i_n)  = \; \mbox{single point}$
whenever the sequence $(i_1,i_2, i_3, \ldots )$ has an infinite number of ones and zeros (i.e., when $(i_1,i_2, i_3, \ldots )\neq (i_1, \ldots i_k, \bar{0}), (i_1, \ldots ,i_k,\bar{1}$). In these cases,  $\Phi(\sigma, \tau_0, \tau_1)$ is a function.
\end{theorem}

\begin{proof} We will show this for $(e,12,12)$, as the others are similar.  We have
\[G_0(e,12,23)=\left(\begin{array}{ccc}
0&0&1/2\\
0&1&0\\
1&0&1/2\end{array}\right)\]
This gives us the following:
\begin{align*}
\tau_n&=\rho_{n-1}\\
\rho_n&\leq\max(\tau_{n-1},\rho_{n-1})\\
\mu_n&=\frac{1}{2}\mu_{n-1}
\end{align*}
We also have
\[G_1(e,12,12)=\left(\begin{array}{ccc}
0&1&1/2\\
1&0&0\\
0&0&1/2\end{array}\right)\]
giving us
\begin{align*}
\tau_n&=\tau_{n-1}\\
\rho_n&=\frac{1}{2}\mu_{n-1}\\
\mu_n&\leq\max(\tau_{n-1},\rho_{n-1})
\end{align*}

If there are an infinite number of zeros and ones, then we know there are infinitely many instances of 01.  Suppose $i_n=0$ and $i_{n+1}=1$.  Then
\begin{align*}
\tau_n&=\rho_{n-1}\\
\rho_n&\leq\max(\tau_{n-1},\rho_{n-1})\\
\mu_n&=\frac{1}{2}\mu_{n-1}
\end{align*}
and
\begin{align*}
\tau_{n+1}&=\tau_n=\rho_{n-1}\\
\rho_{n+1}&=\frac{1}{2}\mu_n=\frac{1}{4}\mu_{n-1}\\
\mu_{n+1}&\leq\max(\tau_n,\rho_n)\leq\max(\tau_{n-1},\rho_{n-1})
\end{align*}

Then we have \[\lim_{n\to\infty}\tau_n=\lim_{n\to\infty}\rho_n=\frac{1}{4}\lim_{n\to\infty}\mu_n=0\]
which implies convergence to a point.  However, in situations where we have only a finite number of zeros or a finite number of ones we do not get convergence.
\end{proof}

\begin{theorem}  For any triangle partition map $(\sigma,\tau_0,\tau_1)$  in  the class $ (e,12,e),$ 
we have that $ \lim_{n\rightarrow \infty} \Gamma_{(\sigma,\tau_0,\tau_1)}(i_1,\dots,i_n) $ never converges to a point,  in which case,  $\Phi(\sigma, \tau_0, \tau_1)$ is a function.
\end{theorem}
 
 \begin{proof}  This will follow from work that we will do in section \ref{degenerate}.  
 
 \end{proof}

\section{Results} \label{results}
In this chapter we show five classes for which $\Phi_{(\sigma,\tau_0,\tau_1)}$ is singular, and exhibit strong evidence for two other classes to being  singular.  Four of these classes are shown through direct computation and a notion of what the "usual" Farey  sequence of a point looks like. The fifth class are what we called M\"{o}nkemeyer-type maps, for which their singularity of the associated $\Phi_{(\sigma,\tau_0,\tau_1)}$ comes directly from Panti's work in \cite{panti}, which in turn uses special properties of these TRIP maps that do not hold for TRIP maps in general.

   The final two (conjectured) classes involve the ergodic properties of the TRIP maps in these classes. 

\subsection{Four classes of singular maps via ``normality''}
In this section, we prove singularity for the associated $\Phi_{(\sigma,\tau_0,\tau_1)}$ for four  classes of TRIP maps through a computational approach.   We begin by presenting a clean framework for showing when $\Phi_{(\sigma,\tau_0,\tau_1)}$ is singular.


\subsubsection{On the importance of $\lim_{n\to\infty}\frac{s_n}{n}$}
We will need the following lemma, whose proof is a geometric fact:
\begin{lemma}\label{area}
The area of a triangle $T$ whose vertices are given in projective coordinates by the $3\times 3$ matrix $M$ with projective coordinates $x,y$ and $z$ being the entries in the top row of $M$ is given by\[|T|=\frac{1}{2}\frac{|\det(M)|}{xyz}\]
\end{lemma}
We also need the following fact about our barycentric partition, whose proof is again a calcuation:
\begin{lemma}\label{half} The area of $\Gamma_{(\sigma, \tau_0, \tau_1)}  (i_1,\dots,i_n)$ is half the area of $\Gamma_{(\sigma, \tau_0, \tau_1)} (i_1,\dots,i_{n-1})$.
\end{lemma}

We are now ready to understand what a typical Barycentric $(\sigma, \tau_0, \tau_1)$ TRIP sequence looks like.

\begin{prop} Let $(\sigma,\tau_0,\tau_1)\in S_3^3$.  Consider a point $(x,y)\in\triangle$ with Barycentric $(\sigma, \tau_0, \tau_1)$ TRIP sequence $(i_1,i_2,\dots)$.  Then for almost all $(x,y)\in\triangle$ we have
\[\lim_{n\to\infty}\frac{\#\{i_k=1,1\leq k\leq n\}}{n}=\frac{1}{2}\]
That is to say, the set of points with normal Barycentric $(\sigma, \tau_0, \tau_1)$ TRIP sequence  has measure one under Lebesgue measure.
\end{prop}
\begin{proof}
(For this proof, we suppress in the notation the  `` $(\sigma, \tau_0, \tau_1)$.'')
By Lemma \ref{half}, $\Gamma(i_1,\dots,i_n)$ has half the area of $\Gamma(i_1,\dots,i_{n-1})$.  Then given a point $(x,y)\in\triangle(i_1,\dots,i_{n-1})$, the probability that $i_n=1$ is 0.5.  Then by the central limit theorem, we have that
\[\lim_{n\to\infty}\frac{\#\{i_k=1,1\leq k\leq n\}}{n}=\frac{1}{2}\]
almost everywhere in $\triangle$.
\end{proof}
Then in order to show that a given $\Phi_{(\sigma,\tau_0,\tau_1)}$ is singular, we simply need to show that the set of points that have normal  Farey $(\sigma, \tau_0, \tau_1)$ TRIP sequence $(i_1,i_2,\dots)$  has measure 0.

We want a way of getting at the proportion of 1's and 0's in the Farey $(\sigma, \tau_0, \tau_1)$ TRIP sequence  for a random $(x,y)\in\triangle_F$.  It turns out that we can do this using the multiplicative version of TRIP maps, discussed in subsection \ref{multiplicative}.  Consider a point $(x,y)\in$ with multiplicative TRIP sequence $(a_1,a_2,a_3,\dots)$.  Define
 \[s_n=a_1+\cdots+a_n.\]  
 
Then we have that 
\[\lim_{n\to\infty}\frac{\#\{i_k=1,1\leq k\leq n\}}{n}=\frac{1}{2}\]
is equivalent to 
\[\lim_{n\to\infty}\frac{s_n}{s_n+n}=\frac{1}{2}.\]
Inverting, this is the same as
\[\lim_{n\to\infty}\frac{s_n+n}{s_n}=\lim_{n\to\infty}1+\frac{n}{s_n}=2\]
which is true if and only if
\[\lim_{n\to\infty}\frac{s_n}{n}=1.\]
Then $\Phi_{(\sigma,\tau_0,\tau_1)}$ is singular if, almost everywhere in $\triangle$, 
\[\lim_{n\to\infty}\frac{s_n}{n}\neq1.\]

We  find four classes for which this is true, and then, in the following section, show that this is true five  classes under an assumption of ergodicity. 

\subsubsection{Understanding $\lim_{n\to\infty}\frac{s_n}{n}$}

In this section we show directly that  the four classes of maps represented by $(e,e,e)$, $(e,e,12)$, $(e,12,e )$ and $(e,12,12)$ have for each that the set of elements in $\triangle$ having $\lim_{n\to\infty}\frac{s_n}{n}=1$ has measure zero.  We will then show that ergodicity in the case of five other classes gives us that this limit is also infinity, and hence, if any of the maps for one of these classes is ergodic, we will have its Minkowksi question mark function being similar.  The reader should note that this method is very similar to the proof of Theorem 14 of \cite{beaver}.

Given $(\sigma,\tau_0,\tau_1)\in S_3^3$, consider the following set:
\[M:=\{(x,y)\in\triangle:\lim_{n\to\infty}\frac{s_n}{n}<\infty\}\]
Certainly the set of elements for which $\lim_{n\to\infty}\frac{s_n}{n}=1$ is included in the set $M$.
Thus if $\lambda(M) = 0$, we will be done.

 Defining
\[M_N:=\{(x,y)\in\triangle:\forall n\geq1,\frac{s_n}{n}<N\}\]
we note that 
\[M=\bigcup_{N=1}^{\infty}M_N.\]
If we can show that $\lambda(M_N)=0$ for each $N$, we will have that $\lambda(M)=0$.  Calculating the area of $M_N$ is difficult, though, so we define the following set
\[\tilde{M}_N:=\{(x,y)\in\triangle:\forall n, a_n<nN\}\]
Since $0 < a_n < a_1 + \cdots + a_n = s_n$, we have  $M_N\subset\tilde{M}_N$. To get a bound on $\lambda(\tilde{M}_N)$ we recursively define the family of sets $\tilde{M}_N(k)$ by
\begin{align*}
\tilde{M}_N(1)&:=\{(x,y)\in\triangle:a_1<N\}\\
\tilde{M}_N(k)&:=\{(x,y)\in\tilde{M}_N(k-1):a_k<kN\}
\end{align*}
Then
\[\tilde{M}_N=\bigcap_{n=1}^{\infty}\tilde{M}_N(k)\]
Our goal is to show, for $(e,e,e)$, $(e,e,12)$, $(e,12,e )$ and $(e,12,12)$,  that $\lambda(\tilde{M}_N(k))\leq\frac{c(k)-1}{c(k)}\lambda(\tilde{M}_N(k-1))$ for some constant $c(k)$ that is linear in $kN$.  If this is the case, then $\lambda(\tilde{M}_N)=0$, which follows from

\begin{lemma} Suppose $\lambda(\tilde{M}_N(k))\leq\frac{akN+c-1}{akN+c}\lambda(\tilde{M}_N(k-1))$ for some positive constants $c$ and $d$.  Then $\lambda(\tilde{M}_N)=0$.
\end{lemma}
\begin{proof}
Assuming the hypothesis and using the fact that $\tilde{M}_N=\bigcap_{k=1}^\infty\tilde{M}_N(k)$ we have that 
\[\lambda(\tilde{M}_N)\leq\prod_{k=2}^\infty\frac{akN+c-1}{akN+c}\]
Showing this product is 0 is equivalent to showing that its reciprocal 
\[\prod_{k=2}^\infty\frac{akN+c}{akN+c-1}=\prod_{k=1}^\infty\left(1+\frac{1}{akN+c-1}\right)=\infty\]
Taking logarithms, this is the same as showing that the series
\[\sum_{k=2}^\infty\log\left(1+\frac{1}{(aN)k+c-1}\right)=\infty\]
which follows by the integral test.  Then we are done.
\end{proof}

  Consider $\tilde{M}_N(k-1)$.  We know this will be made up of the subtriangles of the form $VF_1^{a_1}F_0\cdots F_1^{a_{k}-1}F_0$ where each $a_i<iN$.  Consider one of these subtriangles and denote it $T$.  We define 
\[T_k:=\{(x,y\in T:a_k\geq kN\}\]
Then 
\[\tilde{M}_N(k)=\bigcup_{T\in\tilde{M}_N(k-1)} (T-T_{k}).\]
Now given $T$, let $x,y$ and $z$ denote the projective coordinates of its vertices.  As we will see, it turns out that we can find formulas for $F_1^n$ for the  triangle maps $(e,e,e)$, $(e,e,12)$, $(e,12,e )$ and $(e,12,12)$. Using these formulas, we can get an expression for the area of $T_k$ in terms of $x,y,z$ and $k$.  Ideally, we get some that look like 
\[|T_k|\geq\frac{1}{c(k)}\frac{1}{xyz}\]
where $c(k)$ is a linear function with respect to $k$.  Then we get that 
\[|T-T_k|\leq\frac{c(k)-1}{c(k)}|T|\]
which implies that 
\[\lambda(\tilde{M}_N(k))\leq\frac{c(k)-1}{c(k)}\lambda(\tilde{M}_N(k-1)).\]
Then, if $c(k)$ is linear in $k$, we have shown that this TRIP map gives rise to a singular function.  We will now go through the process of computing $c(k)$ for the four classes mentioned above.  The reader should note that these calculations all follow the same structure, but we present them all for sake of completeness.

We begin with $(e,e,e)$.  (We will suppress the $(e,e,e)$ in the various $F_0(e,e,e)$ and $F_1(e,e,e).$)  We have
\[F_1^k=\left(\begin{array}{ccc}
1&0&k\\
0&1&0\\
0&0&1\end{array}\right)\]
\[F_1^kF_0=\left(\begin{array}{ccc}
0&k&k+1\\
1&0&0\\
0&1&1\end{array}\right)\]
This gives us the following recurrence relations:
\begin{align*}
x_n&=y_{n-1}\\
y_n&=a_nx_{n-1}+z_{n-1}\\
z_n&=(a_n+1)x_{n-1}+z_{n-1}
\end{align*}
Then $z_n\geq y_n\geq x_n$ for all $n$.  Then considering a subtriangle $T$ represented by a matrix whose top row is $(x,y,z)$ we have 
\[\text{area}(T_k)=\frac{1}{xy(kNx+z)}\geq\frac{1}{kN+1}\frac{1}{xyz}\]
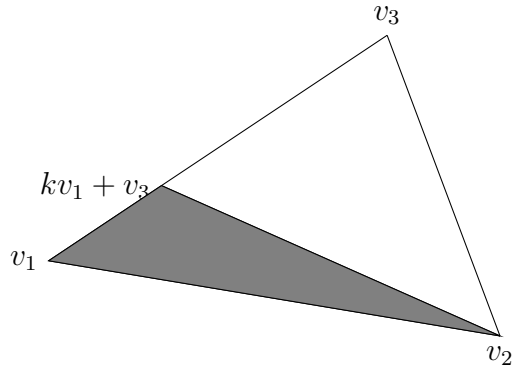
\begin{center}
\begin{tikzpicture}[scale=3]
\draw (0,0)--(2,-1/3);
\draw (0,0)--(3/2,1);
\draw (2,-1/3)--(3/2,1);
\draw (2,-1/3)--(1/2,1/3);
\node [left] at (0,0) {$v_1$};
\node [below] at (2,-1/3) {$v_2$};
\node [above] at (3/2,1) {$v_3$};
\node [left] at (1/2,1/3) {$kv_1+v_3$};
\draw[fill=gray]  (0,0) -- (2,-1/3) -- (1/2,1/3)-- cycle;
\end{tikzpicture}
\end{center}

For $(e,e,12)$ (where we again suppress the $(e,e,12)$), we have
\begin{align*}
F_1^{2k}&=\left(\begin{array}{ccc}
1&0&k\\
0&1&k\\
0&0&1\end{array}\right)&F_1^{2k+1}&=\left(\begin{array}{ccc}
0&1&k+1\\
1&0&k\\
0&0&1\end{array}\right)\\
F_1^{2k}F_0&=\left(\begin{array}{ccc}
0&k&k+1\\
1&k&k\\
0&1&1\end{array}\right)&F_1^{2k+1}F_0&=\left(\begin{array}{ccc}
1&k+1&k+1\\
0&k&k+1\\
0&1&1\end{array}\right)\end{align*}
This gives us the following recurrence relations:
\begin{align*}
a_n&=2k_n&a_n&=2k_n+1\\
x_n&=y_{n-1}&x_n&=x_{n-1}\\
y_n&=k_n(x_{n-1}+y_{n-1})+z_{n-1}&y_n&=k_n(x_{n-1}+y_{n-1})+x_{n-1}+z_{n-1}\\
z_n&=k_n(x_{n-1}+y_{n-1})+x_{n-1}+z_{n-1}&z_n&=(k_n+1)(x_{n-1}+y_{n-1})+z_{n-1}
\end{align*}
Then $z_n\geq y_n\geq x_n$ for all $n$.  Then considering a subtriangle $T$ represented by a matrix whose top row is $(x,y,z)$ we have 
\[\text{area}(T_k)\geq\frac{1}{xy((kN+1)/2(x+y)+z)}\geq\frac{1}{kN+2}\frac{1}{xyz}\]

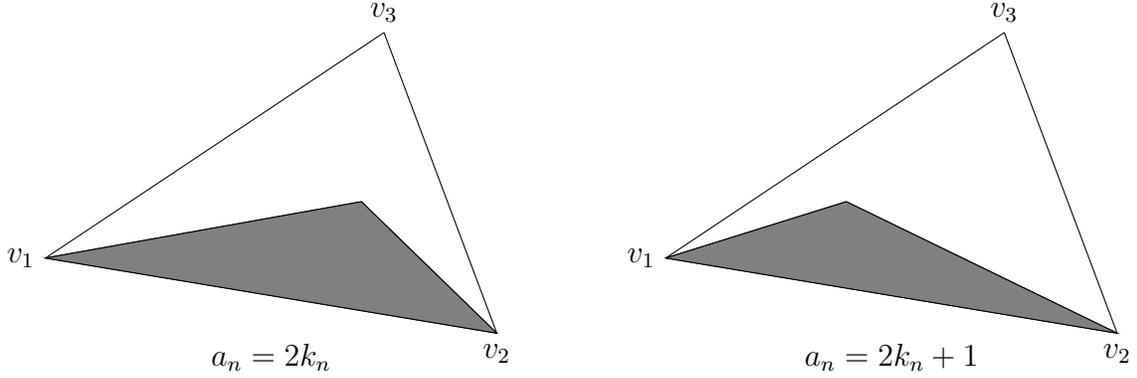
\begin{align*}
\begin{tikzpicture}[scale=3]
\draw (0,0)--(2,-1/3);
\draw (0,0)--(3/2,1);
\draw (2,-1/3)--(3/2,1);
\draw (2,-1/3)--(7/5,1/4);
\draw (0,0)--(7/5,1/4);
\node [left] at (0,0) {$v_1$};
\node [below] at (2,-1/3) {$v_2$};
\node [above] at (3/2,1) {$v_3$};
\draw[fill=gray]  (0,0) -- (2,-1/3) -- (7/5,1/4)-- cycle;
\node [below] at (1,-1/3) {$a_n=2k_n$};
\end{tikzpicture}
&&
\begin{tikzpicture}[scale=3]
\draw (0,0)--(2,-1/3);
\draw (0,0)--(3/2,1);
\draw (2,-1/3)--(3/2,1);
\draw (2,-1/3)--(4/5,1/4);
\draw (0,0)--(4/5,1/4);
\node [left] at (0,0) {$v_1$};
\node [below] at (2,-1/3) {$v_2$};
\node [above] at (3/2,1) {$v_3$};
\draw[fill=gray]  (0,0) -- (2,-1/3) -- (4/5,1/4)-- cycle;
\node [below] at (1,-1/3) {$a_n=2k_n+1$};
\end{tikzpicture}
\end{align*}

For $(e,12,e)$ (and again  we  suppress the $(e,e,12)$)  we have
\[F_1^k=\left(\begin{array}{ccc}
1&0&k\\
0&1&0\\
0&0&1\end{array}\right)\]
\[F_1^kF_0=\left(\begin{array}{ccc}
k&0&k+1\\
0&1&0\\
1&0&1\end{array}\right)\]
This gives us the following recurrence relations:
\begin{align*}
x_n&=a_nx_{n-1}+z_{n-1}\\
y_n&=y_{n-1}\\
z_n&=(a_n+1)x_{n-1}+z_{n-1}
\end{align*}
Then $z_n\geq x_n\geq y_n$ for all $n$.  Then considering a subtriangle $T$ represented by a matrix whose top row is $(x,y,z)$ we have 
\[\text{area}(T_k)=\frac{1}{xy(kNx+z)}\geq\frac{1}{kN+1}\frac{1}{xyz}\]

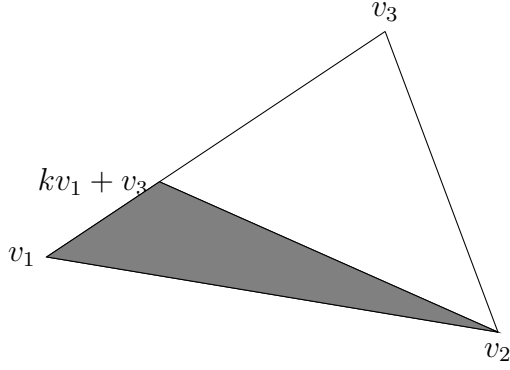
\begin{center}
\begin{tikzpicture}[scale=3]
\draw (0,0)--(2,-1/3);
\draw (0,0)--(3/2,1);
\draw (2,-1/3)--(3/2,1);
\draw (2,-1/3)--(1/2,1/3);
\node [left] at (0,0) {$v_1$};
\node [below] at (2,-1/3) {$v_2$};
\node [above] at (3/2,1) {$v_3$};
\node [left] at (1/2,1/3) {$kv_1+v_3$};
\draw[fill=gray]  (0,0) -- (2,-1/3) -- (1/2,1/3)-- cycle;
\end{tikzpicture}
\end{center}

For $(e,12,12)$ (with us again suppressing  the $(e,12,12)$), we have
\begin{align*}
F_1^{2k}&=\left(\begin{array}{ccc}
1&0&k\\
0&1&k\\
0&0&1\end{array}\right)&F_1^{2k+1}&=\left(\begin{array}{ccc}
0&1&k+1\\
1&0&k\\
0&0&1\end{array}\right)\\
F_1^{2k}F_0&=\left(\begin{array}{ccc}
k&0&k+1\\
k&1&k\\
1&0&1\end{array}\right)&F_1^{2k+1}F_0&=\left(\begin{array}{ccc}
k+1&1&k+1\\
k&0&k+1\\
1&0&1\end{array}\right)\end{align*}
This gives us the following recurrence relations:
\begin{align*}
a_n&=2k_n&a_n&=2k_n+1\\
x_n&=k_n(x_{n-1}+y_{n-1})+z_{n-1}&x_n&=k_n(x_{n-1}+y_{n-1})+x_{n-1}+z_{n-1}\\
y_n&=y_{n-1}&y_n&=x_{n-1}\\
z_n&=k_n(x_{n-1}+y_{n-1})+x_{n-1}+z_{n-1}&z_n&=(k_n+1)(x_{n-1}+y_{n-1})+z_{n-1}
\end{align*}
Then $z_n\geq x_n\geq y_n$ for all $n$.  Then considering a subtriangle $T$ represented by a matrix whose top row is $(x,y,z)$ we have 
\[\text{area}(T_k)\geq\frac{1}{xy((kN+1)/2(x+y)+z)}\geq\frac{1}{kN+2}\frac{1}{xyz}\]

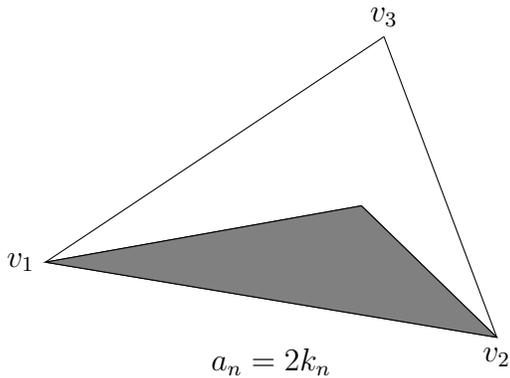
\begin{align*}
\begin{tikzpicture}[scale=3]
\draw (0,0)--(2,-1/3);
\draw (0,0)--(3/2,1);
\draw (2,-1/3)--(3/2,1);
\draw (2,-1/3)--(7/5,1/4);
\draw (0,0)--(7/5,1/4);
\node [left] at (0,0) {$v_1$};
\node [below] at (2,-1/3) {$v_2$};
\node [above] at (3/2,1) {$v_3$};
\draw[fill=gray]  (0,0) -- (2,-1/3) -- (7/5,1/4)-- cycle;
\node [below] at (1,-1/3) {$a_n=2k_n$};
\end{tikzpicture}
&&
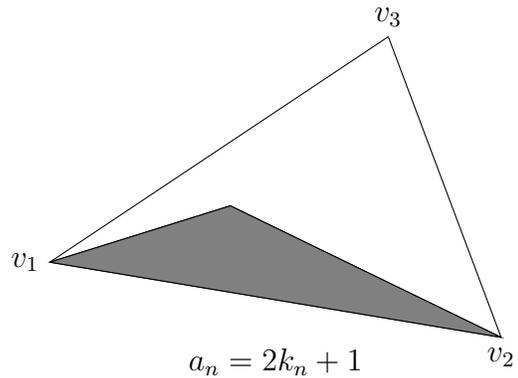
\begin{tikzpicture}[scale=3]
\draw (0,0)--(2,-1/3);
\draw (0,0)--(3/2,1);
\draw (2,-1/3)--(3/2,1);
\draw (2,-1/3)--(4/5,1/4);
\draw (0,0)--(4/5,1/4);
\node [left] at (0,0) {$v_1$};
\node [below] at (2,-1/3) {$v_2$};
\node [above] at (3/2,1) {$v_3$};
\draw[fill=gray]  (0,0) -- (2,-1/3) -- (4/5,1/4)-- cycle;
\node [below] at (1,-1/3) {$a_n=2k_n+1$};
\end{tikzpicture}
\end{align*}

Then $c(k)$ is linear for the classes represented by $(e,e,e)$, $(e,e,12)$, $(e,12,e)$, and $(e,12,12)$.  This means that each of these classes gives rise to a singular function.  We list out the maps in $(e,e,e)$, $(e,e,12)$ and $(e,12,12)$ below, leaving the class $(e,12,e)$ to be treated in greater detail in the next section.
\begin{align*}
&(e,e,e)&&(12,12,12)&&(13,13,13)&&(23,23,23)&&(123,132,132)&&(132,123,123)\\
&(13,12,12)&&(123,e,e)&&(e,123,123)&&(132,132,132)&&(12,23,23)&&(13,13,13)\\
&(e,e,12)&&(12,12,e)&&(13,13,123)&&(23,23,132)&&(123,132,23)&&(132,123,13)\\
&(13,e,12)&&(123,12,e)&&(e,13,123)&&(132,23,132)&&(12,132,23)&&(13,123,13)\\
&(e,e,123)&&(12,12,23)&&(13,13,12)&&(23,23,13)&&(123,132,e)&&(132,123,132)\\
&(e,13,12)&&(12,132,12)&&(13,e,123)&&(23,123,132)&&(123,12,23)&&(132,23,13)\\
&(e,12,12)&&(12,e,e)&&(13,123,123)&&(23,132,132)&&(123,23,23)&&(132,13,13)\\
&(13,e,e)&&(123,12,12)&&(e,13,13)&&(132,23,23)&&(12,132,132)&&(13,123,123)\\
\end{align*}

We can go through similar calculations as above to attempt to calculate $c(k)$ for any of the 108 polynomial TRIP maps.  Unfortunately, for the remaining maps not in these classes we run into one of two problems.  Either $c(k)$ is quadratic, which does not give us what we want, or we can't actually calculate $c(k)$.

\subsection{Degenerate Farey Maps}\label{degenerate}
We have just shown that $\Phi(e, 12, e)$ is singular.  There is another method for showing this singularness. 
As mentioned earlier, it is certainly not the case that the nested triangles $\triangle_{\sigma, \tau_0, \tau_1)}(i_1, \ldots , i_n) $ converge to a point.  There are some maps for which this nested sequence will never converge to a point, namely for what we call {\it degenerate TRIP maps}.
Degenerate TRIP maps fix one of the original vertices of $\triangle$ and partition the opposite side according to the same Farey division of the unit interval.  The TRIP map for $(e, 12, e)$ is degenerate, as both 
$$F_0(e, 12, e) = \left(\begin{array}{ccc} 0 & 0 & 1 \\ 0 & 1 & 0 \\ 1 & 0 & 1 \end{array} \right) \; \mbox{and} \; F_1(e, 12, e) = \left(\begin{array}{ccc} 1 & 0 & 1 \\ 0 & 1 & 0 \\ 0 & 0 & 1 \end{array} \right) , $$
meaning that 
$$(v_1, v_2, v_3 )F_0(e, 12, e) =( v_3, v_2, v_1 + v_3) \; \mbox{and}  \; (v_1, v_2, v_3 )F_1(e, 12, e) = (v_1, v_2,  v_1 +  v_3 ),$$
 leaving the vertex $v_2$ fixed,
 By computation, one can show there are  three possible partitionings, namely :
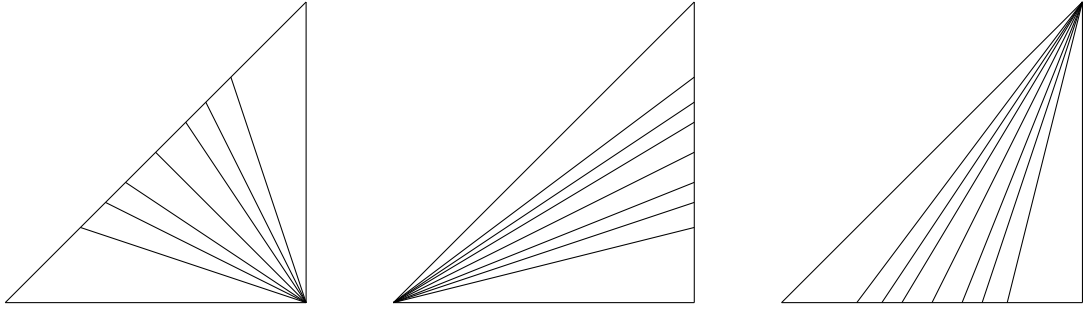
\begin{align*}
&\begin{tikzpicture}[scale=4]
\draw (0,0)--(1,1);
\draw (0,0)--(1,0);
\draw (1,0)--(1,1);
\draw (1,0)--(1/2,1/2);
\draw (1,0)--(1/3,1/3);
\draw (1,0)--(2/3,2/3);
\draw (1,0)--(1/4,1/4);
\draw (1,0)--(3/4,3/4);
\draw (1,0)--(2/5,2/5);
\draw (1,0)--(3/5,3/5);
\end{tikzpicture}&&
\begin{tikzpicture}[scale=4]
\draw (0,0)--(1,1);
\draw (0,0)--(1,0);
\draw (1,0)--(1,1);
\draw (0,0)--(1,1/2);
\draw (0,0)--(1,1/3);
\draw (0,0)--(1,2/3);
\draw (0,0)--(1,1/4);
\draw (0,0)--(1,3/4);
\draw (0,0)--(1,2/5);
\draw (0,0)--(1,3/5);
\end{tikzpicture}&&
\begin{tikzpicture}[scale=4]
\draw (0,0)--(1,1);
\draw (0,0)--(1,0);
\draw (1,0)--(1,1);
\draw (1,1)--(1/2,0);
\draw (1,1)--(1/3,0);
\draw (1,1)--(2/3,0);
\draw (1,1)--(1/4,0);
\draw (1,1)--(3/4,0);
\draw (1,1)--(2/5,0);
\draw (1,1)--(3/5,0);
\end{tikzpicture}&
\end{align*}
The TRIP map for $(e, 12, e)$, for which we have already shown that $\Phi(e, 12, e)$ is singular, is degenerate.

The $\Phi(\sigma, \tau_0, \tau_1)$ associated with these classes of maps are nice because we can define them in terms of the original Question Mark Function.

For the permutations that fix $(0,0)$, we have that $$\Phi(1,y)=(1,?(y)), \Phi(x,0)=(x,0), \Phi(x,x)=(x,x),$$ and that any point $(x,y)$ on the interior of $\triangle$ on the line with slope $\alpha$ will be sent to the point on the line with slope $?(\alpha)$ that the same proportion of the distance along the line.

For the permutations that fix $(1,0)$, we have that $$\Phi(x,x)=(?(x),?(x)), \Phi(x,0)=(x,0), \Phi(1,y)=(1,y),$$ and that any point $(x,y)$ on the interior of $\triangle$ on the line that passes through $(\alpha,\alpha)$ will be sent to appropriate point on the line passing through $(?(\alpha),?(\alpha))$.

For the permutations that fix $(1,1)$, we have that $$\Phi(x,0)=(?(x),0), \Phi(x,x)=(x,x), \Phi(1,y)=(1,y),$$ and that any point $(x,y)$ on the interior of $\triangle$ on the line that passes through $(\alpha,0)$ will be sent to appropriate point on the line passing through $(?(\alpha),0)$.

By the singularity of $?(x)$, we have that $\Phi_{(\sigma,\tau_0,\tau_1)}$ for any $(\sigma,\tau_0,\tau_1)$ belonging to the degenerate Farey class will be singular.  To see this, consider the permutation $(e,12,e)$ which gives a Farey partitioning of the hypotenuse of $\triangle$.  Then take the set of points $(A)$ of measure 1 on which $?(x)$ is singular, meaning the measure of $?(A)$ is 0.  Then the set of line segments connecting the points in the form $(a,a)$ where $a\in A$ on the hypotenuse to the vertex $(1,0)$ will have full measure on $\triangle$. Denote this set as $B$.  The $\Phi_{(e,12,e)}(B)$ will have measure 0 because $?(A)$ has measure 0.  Then because $\Phi{(e,12,e)}$ is singular we have that the remaining maps in this class are singular, which we list below:
\begin{align*}
&(e,12,e)&&(e,12,13)&&(e,123,e)&&(e,123,13)\\
&(13,12,e)&&(13,12,13)&&(13,123,e)&&(13,123,13)\\
&(12,e,12)&&(12,e,132)&&(12,23,12)&&(12,23,132)\\
&(123,e,12)&&(123,e,132)&&(123,23,12)&&(123,23,132)\\
&(23,13,23)&&(23,13,123)&&(23,132,23)&&(23,132,123)\\
&(132,13,23)&&(132,13,123)&&(132,132,23)&&(132,132,123)
\end{align*}

\subsection{The ergodic cases}

In Section \ref{multiplicative}, each $(\sigma, \tau_0, \tau_1) \in S_3^3$ defines the multiplicative triangle  partition map

\begin{prop}\label{ergodic} If the multiplicative triangle partition map  $T^G(\sigma, \tau_0, \tau_1)$ is ergodic and the associated $\triangle_k (\sigma, \tau_0, \tau_1)$ satisfy that the area of $\triangle_k (\sigma, \tau_0, \tau_1)$ is $\frac{1}{(k+1)(k+2)}$, then 
\[\lim_{n\to\infty}\frac{s_n}{n}=\infty\]
almost everywhere.
\end{prop}
\begin{proof}
For ease of notation, suppress the triple $(\sigma, \tau_0, \tau_1)$.  
First we define $f_k$ as the characteristic function of $\triangle_k$:
\[f_k(x)=\begin{cases}1&x\in\triangle_k\\0&x\notin\triangle_k\end{cases}\]
Then using the fact that every ergodic multiplicative triangle map has an intrinsic invariant measure $\mu$, by Birkhoff's Ergodic Theorem we have that
\[\lim_{n\to\infty}\frac{1}{n}\sum_{i=1}^nf_k(T^i(x))=\int_\triangle f_k(x)d\mu=\mu (\triangle_k )\]
for almost all $x\in\triangle$.

Using the notation $P(k)=\mu(\triangle_k)$, this says that for almost all $x\in\triangle$, $a_i=k$ on average $P(k)$ of the time.  Then we have
\[\lim_{n\to\infty}\frac{s_n}{n}=\sum_{k=1}^\infty kP(k)\]
almost everywhere. Because the intrinsic measure of the multiplicative triangle partition map is absolutely continuous with respect to the Lebesgue measure, we have that $P(k)>C\frac{1}{(k+1)(k+2)}$ for some constant $C$.  Using the integral test, we have that
\[\int xP(x)dx\geq\int\frac{Cx}{(x+1)(x+2)}dx=C[\log\left(\frac{(x+2)^2}{x+1}\right)+c_0]\]
which diverges on $[1,\infty)$.  Thus we get
\[\lim_{n\to\infty}\frac{s_n}{n}=\infty\]
almost everywhere.
\end{proof}
Either by direct calculation, or by looking in \cite{amburg}, 
 we see that the  classes  
\begin{align*}
&(e,13,e)&&(e,13,23)&&(e,13,132)&&(e,23,e)&&(e,123,132)
\end{align*}
contain a map that satisfies the  condition that the area of $\triangle_k (\sigma, \tau_0, \tau_1)$ is $\frac{1}{(k+1)(k+2)}$.

In the case of $(e,13,e)$ and $(e,23,e)$ the map that satisfies the area requirement on the $\triangle_k (\sigma, \tau_0, \tau_1)$ is the class representative itself.  For the class $(e,13,23)$ this map is $(132,12,123)$.  For the class $(e,13,132)$ this map is $(13,23,123)$.  For the class $(e,123,132)$ this map is $(13,23,13)$.  Then if we can show these maps are ergodic we will get these five classes as well. 

In current work of Amburg and Jensen \cite{ergodic}, they believe that  $(e,23,e)$ and $(132,12,123)$ are ergodic, which would give us  that the following 24 maps give rise to singular $\Phi(\sigma, \tau_0, \tau_1)$:
\begin{align*}
&(e,23,e)&&(12,123,12)&&(13,132,13)&&(23,e,23)&&(123,13,132)&&(132,12,123)\\
&(13,12,132)&&(123,e,13)&&(e,123,23)&&(132,132,12)&&(12,23,123)&&(23,13,e)\\
&(e,13,23)&&(12,123,132)&&(13,e,123)&&(23,132,e)&&(123,12,13)&&(132,23,12)\\
&(13,132,123)&&(123,13,23)&&(e,23,12)&&(132,12,13)&&(12,123,e)&&(23,e,132)\\
\end{align*}   

Note also, that   Messaoudi,  Nogueira, and  Schweiger \cite{SchweigerF08} have shown earlier that $T^G(e,e,e)$ is ergodic, and hence we have, to a slight extent, another way of showing that $\Phi(e,e,e)$ is singular. 

\subsection{A Special Case: M\"{o}nkemeyer-Type Maps}\label{Monkmayer}

As we mentioned in the introduction, Panti's generalization of the question mark function comes from his recognizing that $?(x)$ is completely characterized by the fact that it is the unique homomorphism that conjugates the Farey map with the Tent map.  Thus, given that the $n$-dimensional M\"{o}nkemeyer map is the generalization of the Farey map, it's only natural that his function is the unique homorphism that conjugates the $n$-dimensional M\"{o}nkemeyer map with the $n$-dimensional Tent map.

In \cite{panti}, Panti explicitly shows that his function is singular under the traditional measure theoretic definition of singularity.  From \cite{tripmaps} we know that the M\"{o}nkemeyer map in 2-dimensions corresponds to the permutation triple $(e,23,132)$.  Then by a combination of Corollary \ref{permute} and Lemma \ref{monk1} we have that the following permutations give rise to a singular function:

\begin{align*}
&(e,23,23)&&(e,23,132)&&(e,132,23)&&(e,132,132)\\
&(13,23,23)&&(13,23,132)&&(13,132,23)&&(13,132,132)\\
&(12,13,13)&&(12,13,123)&&(12,123,13)&&(12,123,123)\\
&(123,13,13)&&(123,13,123)&&(123,123,13)&&(123,123,123)\\
&(23,e,e)&&(23,e,12)&&(23,12,e)&&(23,12,12)\\
&(132,e,e)&&(132,e,12)&&(132,12,e)&&(132,12,12)
\end{align*}

What's nice about the M\"{o}nkemyer-type maps is that every triangle sequence corresponds to a unique point. This allows us to construct a bijection between the points of $\triangle$ and the points of $\{0,1\}^\N$ modulo an equivalence relation.  This is what allows Panti to set up the various conjugations that he uses to both define $\Phi$ and prove its singularity.  Unfortunately, this is not the case for all TRIP maps, forcing us  to our alternative approaches.

\section{Conclusion}

There are many questions left.  The next immediate step would be to say something about the 60 triangle partition maps that fall into one of the five classes
\begin{align*}
&(e,e,13)&&(e,e,23)&&(e,e,132)&&(e,12,23)&&(e,12,132).
\end{align*}

 Also, in the introduction, we stated that the family of triangle partition maps include most know multi-dimensional continued fraction algorithms.  Our 216 triangles maps are the generators of the family, but do not themselves capture most multi-dimensional continued fractions.  As explained in \cite{tripmaps}, we need to look at combination triangle partition maps, which are various combinations of the 216 maps.  Each of these combination TRIP maps should have an associated Minkowski question mark function.  We strongly suspect that the techniques of this paper could be used to find analogous results for combination TRIP maps that are made up of TRIP maps whose question mark function is known to be singular.  More interesting are those combination TRIP maps that are made up of the TRIP maps that we do know know about.
 
 Over the last century there as been a lot of work on the traditional Minkowski question mark function.  For example, see the bibliography at http://uosis.mif.vu.lt/~alkauskas/minkowski.htm
 prepared by Giedrius Alkauskas.  The topics of most of these papers suggest natural question for multi-dimensional continued fraction algorithms.  Hence this current paper should only be viewed as the beginning of work.


\begin{thebibliography}{1}

\bibitem{Stern}  Ilya Amburg, Krishna Dasaratha, Laure Flapan, Thomas Garrity, Chansoo Lee, Cornelia Mihaila, Nicholas Neumann-Chun, Sarah Peluse, Matthew Stroffregen, Stern Sequences for a Family of Multidimensional Continued Fractions: 
TRIP-Stern Sequences, {\it Journal of Integer Sequences}, article 17.1.7.

\bibitem{amburg} Ilya Amburg and Thomas Garrity, Functional Analysis Behind a Family of Multidimensional Continued Fractions: Triangle Partition Maps,  in preparation.

\bibitem{ergodic} Ilya Amburg and Stephanie Jensen, Ergodicity for Select Triangle Partition Maps, in preparation.



\bibitem{beaver} Olga R. Beaver and Thomas Garrity, A Two-Dimensional Minkowski ?(x) Function, Journal of Number Theory, 107(1):105-134, 2004.



\bibitem{GarrityT05} L. Chen, T. Cheslack-Postava, B. Cooper, A. Diesl, T. Garrity, M. Lepinski, and A. Schuyler, A dual approach to triangle sequences: a multidimensional continued fraction algorithm, \textit{Integers} {\bf 5} (2005).


\bibitem{tripmaps} K.  Dasaratha,  L.  Flapan, T. Garrity, C. Lee, C.  Mihaila, N.  Neumann-Chun, S.  Peluse, and M.  Stoffregen,  A generalized family of multidimensional continued fractions: triangle partition maps, \textit{Int. J. Number Theory} {\bf 10} (2014), 2151--2186.


\bibitem{SMALL11q3}   K. Dasaratha, L. Flapan, T. Garrity, C. Lee, C. Mihaila, N. Neumann-Chun, S. Peluse, and M. Stoffregen, Cubic irrationals and periodicity via a family of multi-dimensional continued fraction algorithms, \textit{Monatsh. Math.} {\bf 174} (2014), 549--566.


\bibitem{Fogg}  N. Fogg, \textit{Substitutions in Dynamics, Arithmetics and Combinatorics}, Springer, 2002.




\bibitem{GarrityT01} T. Garrity, On periodic sequences for algebraic numbers, \textit{J. Number Theory} {\bf 88} (2001), 86--103.

   
   
\bibitem{Jensen} S. Jensen, \textit{Ergodic Properties of Triangle Partition Maps: a Family of Multidimensional Continued Fractions}, senior thesis, Williams College, 2012.
   
   
   
   
\bibitem{Karpenkov} O. Karpenkov, \textit{Geometry of Continued Fractions}, Springer, 2013.



\bibitem{Lagarias93} J. Lagarias, The quality of the diophantine approximations found by the Jacobi-Perron algorithm and related algorithms, \textit{ Monatsh. Math.} {\bf 115} (1993), 299--328.

\bibitem{marder}  A. Marder, Two-Dimensional Analogs of the Minkowski $?(x)$ function, https://pdfs.semanticscholar.org/8941/21889dda2f96d5bc2fa8a7e5799f27c2f330.pdf

\bibitem{SchweigerF08} A. Messaoudi, A. Nogueira, and F. Schweiger, Ergodic properties of triangle partitions, {\it Monatsh. Math.} {\bf 157} (2009), 283--299.



\bibitem{panti} Giovanni Panti, Multidimensional Continued Fractions and a Minkowski Function, arXiv:0705.0584v2, to appear in {\em Monatshefte fur Mathematik}.

\bibitem{salem} R. Salem, On some singular monotonic functions which are strictly in- creasing, {\em Transactions of the American Mathematical Society}, 53, (1943), 427-439.

\bibitem{schweiger} Fritz Schweiger, {\it Multidimensional Continued Fractions}. Oxford Science Publications. Oxford University, Oxford 2000.

\bibitem{viader} P. Viader, J. Paradis and L. Bibiloni, A New Light on Minkowski's ?(x) Function, {\em Journal of Number Theory}, 73, (1998), 212-227.

\end{thebibliography}
\end{document}